\documentclass[a4paper,12pt,reqno]{amsart}
\usepackage{amsmath,amssymb,amsthm}
\usepackage{latexsym}
\usepackage{color}
\usepackage{graphicx}
\usepackage{mathrsfs}
\usepackage{enumerate}
\usepackage{enumitem}
\usepackage[abbrev]{amsrefs}
\usepackage[T1]{fontenc}
\usepackage{bbm}
\usepackage{mathtools}
\mathtoolsset{showonlyrefs=true}

\usepackage{tikz}
\usetikzlibrary{arrows.meta}
\usepackage{float}

\usepackage[colorlinks,
            linkcolor=blue,      
           anchorcolor=blue,  
          citecolor=red,       
          ]{hyperref}

\allowdisplaybreaks[4]

\theoremstyle{plain}
\newtheorem{thm}{Theorem}[section]
\newtheorem{lemm}[thm]{Lemma}
\newtheorem{prop}[thm]{Proposition}

\theoremstyle{definition}
\newtheorem{df}[thm]{Definition}
\newtheorem{rem}[thm]{Remark}

\makeatletter

\@addtoreset{equation}{section}
\makeatother

\renewcommand{\div}{\operatorname{div}}
\newcommand{\dB}{\dot{B}}

\newcommand{\supp}{\operatorname{supp}}

\renewcommand{\leq}{\leqslant}
\renewcommand{\geq}{\geqslant}

\newcommand{\n}[1]{{\left\|#1\right\|}}

\newcommand{\Id}{{\rm Id}}

\newcommand{\R}{\mathbb{R}}
\newcommand{\T}{\mathbb{T}}
\newcommand{\N}{\mathbb{N}}
\newcommand{\Z}{\mathbb{Z}}

\newcommand{\C}{\mathbb{C}}

\newcommand{\lp}[1]{\left[#1\right]}
\newcommand{\Mp}[1]{\left\{#1\right\}}
\renewcommand{\sp}[1]{\left(#1\right)}

\newcommand{\p}{{\rm p}}
\renewcommand{\r}{{\rm r}}

\newcommand{\bb}{\mathbbm}

\begin{document}
\title[Non-uniqueness for the Navier--Stokes equations]
{Sharp non-uniqueness for the Navier--Stokes equations in scaling critical spaces}
\author{Mikihiro Fujii}
\address{Graduate School of Science, Nagoya City University, Nagoya, 467-8501, Japan}
\email{fujii.mikihiro@nsc.nagoya-cu.ac.jp}
\begin{abstract}
It is known that uniqueness of mild solutions to the incompressible Navier--Stokes equations holds in the critical class $C([0,T);L^n(\mathbb{R}^n))$ for $n \geqslant 3$. 
In this paper, we prove that this result is sharp in the sense that uniqueness fails if $L^n(\mathbb{R}^n)$ is replaced by some scaling critical spaces that are even slightly larger. 
We achieve this through a complete classification for every pair $(p,q)$ of whether uniqueness of mild solutions in the critical Besov class $C([0,T);\dot{B}_{p,q}^{n/p-1}(\mathbb{R}^n))$ holds or not. 
Our non-uniqueness mechanism produces infinitely many global solutions emanating even from zero initial state, whose large-time asymptotics are governed by non-trivial stationary flows.
To the best of our knowledge, such non-unique solutions provide the first examples of non-dissipative unforced Navier--Stokes flow with critical regularity.
\end{abstract}
\maketitle


\section{Introduction}\label{sec:intro}
Let us consider the initial value problem for the incompressible Navier--Stokes equations on $\R^n$ with $n \geq 2$:
\begin{align}\label{eq:NS-intro}
    \begin{cases}
        \partial_t u - \Delta u + \mathbb{P}\div (u \otimes u) = 0, \qquad & t>0, x \in \mathbb{R}^n, \\
        \div u = 0, \qquad & t\geq 0, x \in \mathbb{R}^n, \\
        u(0,x) = u_0(x), & x \in \mathbb{R}^n, \\
    \end{cases}
\end{align}
where $u=u(t,x):[0,\infty) \times \R^n \to \R^n$ is the unknown velocity field of the fluid, whereas $u_0=u_0(x):\R^n \to \R^n$ is a given initial datum satisfying $\div u_0=0$. 
We denote by $\mathbb{P}=\Id+\nabla\div(-\Delta)^{-1}$ the Helmholtz projection onto divergence-free vector fields.

In dimensions $n \geq 3$,
Lions--Masmoudi \cite{Lio-Mas} and Furioli--Rieusset--Terraneo \cite{Fur-Rie-Ter-2000} proved that mild solutions to \eqref{eq:NS-intro} are unique in the critical space $C([0,T);L^n(\R^n))$.
The aim of this paper is to prove the sharpness of the choice of $L^n(\R^n)$ in the sense that the uniqueness fails if one replaces $L^n(\R^n)$ by a scaling critical space that is even slightly larger;
for instance, enlarging $L^n(\R^n)=\dot{F}_{n,2}^0(\R^n)$ to a slightly larger Lizorkin--Triebel space $\dot{F}_{n,q}^0(\R^n)$ with any $q>2$ leads to non-uniqueness of the mild solutions in $C([0,T);\dot{F}_{n,q}^0(\R^n))$. 
To achieve this, 
we show that \emph{the necessary and sufficient condition} for the uniqueness of mild solutions to \eqref{eq:NS-intro} in the scaling critical Besov spaces
is given by $1 \leq p < n$ with $1 \leq q \leq \infty$, or $p=n$ with $1 \leq q \leq 2$; see Figure~\ref{fig:uniq-nonuniq} below.
Moreover, for the non-unique case $p=n$ with $2<q\leq \infty$, or $n < p \leq \infty$ with $1 \leq q \leq \infty$,
we prove that the zero initial datum $u_0=0$ generates infinitely many small global solutions $u \in C([0,\infty);\dB_{p,q}^{n/p-1}(\R^n))$, which asymptotically converge to a non-trivial flow $U \in \dB_{p,q}^{n/p-1}(\R^n)$ solving the unforced stationary Navier--Stokes equations:
\begin{align}\label{eq:sNS}
    \begin{cases}
        -\Delta U + \mathbb{P} \div (U \otimes U) = 0, \qquad & x \in \mathbb{R}^n,\\
        \div U = 0, \qquad & x \in \mathbb{R}^n.
    \end{cases}
\end{align}

\subsection{Known results related to our work}
Before we state our main results precisely, we recall the previous studies related to our work.
By the Duhamel principle, we may formally rewrite
\eqref{eq:NS-intro} as the integral equation
\begin{align}\label{eq:int}
    u(t)
    = 
    e^{t\Delta}u_0
    -
    \int_0^t 
    e^{(t-\tau)\Delta}
    \mathbb{P}\div(u(\tau)\otimes u(\tau))d\tau.
\end{align}
A solution to \eqref{eq:int} is called a mild solution of \eqref{eq:NS-intro}.
In view of the Fujita--Kato principle, 
it is natural to construct mild solutions in scaling critical spaces.
Here, let us recall the notion of scaling criticality. 
If $u$ is a mild solution, then  $u_\lambda(t,x)=\lambda u(\lambda^2 t,\lambda x)$
also solves \eqref{eq:int} for all $\lambda>0$. 
A Banach space $X$
is called a scaling critical space if $\n{u_{\lambda}(0,\cdot)}_X = \n{u(0,\cdot)}_X$ for all $\lambda>0$. 
Typical examples of scaling-critical spaces and their inclusion relations are as follows:
\begin{align}
    \dot{H}^{\frac{n}{2}-1}(\mathbb{R}^n)
    \hookrightarrow 
    L^n(\mathbb{R}^n)
    &
    \hookrightarrow 
    L^{n,\infty}(\mathbb{R}^n)
    \\
    &
    \hookrightarrow
    \dot{B}_{p,\infty}^{\frac{n}{p}-1}(\R^n)
    \hookrightarrow
    {\rm BMO}^{-1}(\R^n)
    \hookrightarrow
    \dB_{\infty,\infty}^{-1}(\R^n)
\end{align}
with $n<p<\infty$ and 
\begin{align}
    \dot{B}_{p,q}^{\frac{n}{p}-1}(\mathbb{R}^n)\hookrightarrow {\rm BMO}^{-1}(\mathbb{R}^n)
\end{align}
with $1\leq p<\infty$ and $1\leq q\leq\infty$.
This line of research originates from the pioneering work of Fujita--Kato \cite{Fuj-Kato-64-ARMA}, where they proved well-posedness
in $\dot{H}^{n/2-1}(\mathbb{R}^n)$, and was extended to $L^n(\mathbb{R}^n)$ by Kato \cite{Kato-84}
and Giga--Miyakawa \cite{Gig-Miy-85}.
Moreover, Kozono--Yamazaki \cite{Koz-Yam-94}, Cannone--Planchon \cite{Can-Pla-96}, and Planchon \cite{Pla-98}
proved well-posedness in
$\dot{B}_{p,q}^{n/p-1}(\mathbb{R}^n)$ for $1\leq p<\infty$ and $1\leq q\leq\infty$.
In the endpoint case $p=\infty$, Koch--Tataru \cite{Koh-Tat-01} established well-posedness for small data in
${\rm BMO}^{-1}(\mathbb{R}^n)$, however, it was shown by \cites{Bou-Pav-08,Yon-10,Wan-15} that \eqref{eq:NS-intro} is ill-posed in $\dot{B}_{\infty,q}^{-1}(\mathbb{R}^n)$ ($1 \leq q \leq \infty$).
See \cite{IN} for more precise analysis on the well-posedness of \eqref{eq:NS-intro} in critical Besov or Lizorkin--Triebel spaces.
We refer to \cites{Kozono-Shimizu-2018,Takeuchi-2025-SIMA,Yam-2000,arXiv:2509.21272v2} for the forced Navier--Stokes flow in critical spaces.

Next, we address the issue of uniqueness for the Navier--Stokes equations. This topic was brought to the forefront by Leray's pioneering work \cite{Leray-1934-Acta}, which produced weak solutions for any divergence-free initial datum in $L^2(\R^3)$ but did not reveal their uniqueness.
It is well-known that the Leray--Hopf weak solutions are unique if they belong to the Ladyzhenskaya--Prodi--Serrin class $L^p(0,T;L^q(\R^n))$ with $n/q+2/p=1$; see \cites{Serrin-1963,Masuda-1984,Escauriaza-Seregin-Sverak-2003,Kozono-Sohr-1996} for more detail.
For the mild solutions, Lions--Masmoudi \cite{Lio-Mas} and Furioli--Rieusset--Terraneo \cite{Fur-Rie-Ter-2000} proved uniqueness in the scaling critical class $C([0,T);L^n(\R^n))$ with $n \geq 3$, which corresponds to an endpoint $(p,q)=(\infty,n)$ of the Ladyzhenskaya--Prodi--Serrin class.
See also \cite{Yam-2000} for the uniqueness of small mild solutions in $L^{\infty}(0,T;L^{n,\infty}(\R^n))$ and \cite{Miu-2005} for the uniqueness in $C([0,T);bmo^{-1}(\R^n)) \cap L^{\infty}_{\rm loc}((0,T);L^{\infty}(\R^n))$.
The recent progress of this direction is given by Zhan \cite{arXiv:2402.01174v2}, where he considered the uniqueness of mild solutions in $C([0,T);\widetilde{L}^{n,\infty}(\R^n))$ with a new space $\widetilde{L}^{n,\infty}(\R^n)=\overline{L^{n,\infty}(\R^n) \cap L^{\infty}(\R^n)}^{L^{n,\infty}}$ which is strictly larger than the completion of $C_c^{\infty}(\R^n)$ in $L^{n,\infty}(\R^n)$.
In terms of the scaling critical Besov spaces, the following proposition holds.
\begin{prop}\label{prop:uniq}
    Let $n \geq 3$.
    Let $p$ and $q$ satisfy either of the following.
    \begin{itemize}
        \item [\rm (U1)] $1 \leq p < n$ and $1 \leq q \leq \infty$,
        \item [\rm (U2)] $p=n$ and $1 \leq q \leq 2$.
    \end{itemize}
    Then, the uniqueness of mild solutions to \eqref{eq:NS-intro} holds in $C([0,T);\dB_{p,q}^{n/p-1}(\R^n))$, provided a smallness condition on solutions only when $q=\infty$.
\end{prop}
The proof of Proposition \ref{prop:uniq} follows immediately by adopting the argument in \cites{Fur-Rie-Ter-2000,Fuj-AIPHC,Iwabuchi-Okazaki-2025} in the case of (U1) and by the result of \cite{Fur-Rie-Ter-2000} with the continuous embedding $\dB_{n,q}^0(\R^n) \hookrightarrow L^n(\R^n)$ holds for the case of (U2).
Note that Proposition \ref{prop:uniq} still holds if one replaces Besov spaces by Lizorkin--Triebel spaces as the continuous embedding $\dot{F}_{p,q}^{n/p-1}(\R^n) \hookrightarrow L^n(\R^n)$ for all $(p,q)$ satisfying (U1) or (U2).

Finally, we briefly review works on non-uniqueness. One major direction originates from applying the Nash convex integration
to the Euler equations on the torus (see for instance \cites{DeLellis-Szekelyhidi-09,Isett-18,Daneri-Szekelyhidi-2017-ARMA,Lellis-Szekelyhidi-2013-Invent,Buckmaster-Lellis-Isett-Szekelyhidi-2015-AnnMath}). Using an iterative scheme based on the Euler--Reynolds system via the geometric lemma,
these works constructed Euler flows violating the energy conservation law and made substantial progress toward the Onsager conjecture.
Although incorporating the viscous term was long thought to be fundamentally difficult, Buckmaster--Vicol \cite{Buc-Vic-19}
introduced the concept of intermittent convex integration and proved that uniqueness of weak solutions to the
Navier--Stokes equations on $\mathbb{T}^3$ breaks down in the supercritical class
$C([0,T);H^\beta(\mathbb{T}^3))$ for some $0<\beta\ll 1/2$.
Cheskidov--Luo \cite{Che-Luo-21} showed that an endpoint Ladyzhenskaya--Prodi--Serrin class
$L^2(0,T;L^\infty(\mathbb{T}^n))$
is optimal with respect to the time integrability exponent; they constructed non-unique solutions in
$L^p(0,T;L^\infty(\mathbb{T}^3))$ for $p<2$.
In \cite{Che-Luo-22-AnnPDE}, they also considered another endpoint in two dimensions and proved non-uniqueness in
$C([0,T);L^p(\mathbb{T}^2))$ for $p<2$.
Furthermore, Cheskidov--Zeng--Zhang \cite{arXiv:2503.05692} claimed that every initial datum in
$H^{1/2}(\mathbb{T}^3)$ generates infinitely many Navier--Stokes flows whose kinetic energy is continuous
and monotonically decreasing in time. 
Remarkably, Coiculescu--Palasek \cite{Coi-Pal-25} showed that even in
${\rm BMO}^{-1}(\mathbb{T}^3)$, where well-posedness for small data is known, a certain large initial datum generates two smooth solutions.
Their argument builds on ideas developed by Palasek \cite{Pal-25-IMRN} for the dyadic Navier--Stokes model.
Pushing this direction further, Cheskidov--Dai--Palasek \cite{arXiv:2511.09556} observed that non-uniqueness, in the form of
instantaneous Type I blow-up, may be generated from arbitrary $C^{\infty}$ solenoidal initial data.
Another strategy to prove non-uniqueness is developed by Jia--Sver\'ak \cites{Jia-Sve-Invent,Jia-Sve-JFA}, and is to exploit the instability of certain self-similar solutions.
A notable result in this direction is due to Albritton--Bru\'e--Colombo \cite{Alb-Bru-Col-AnnMath}, who proved non-uniqueness of Leray--Hopf weak solutions
to the forced Navier--Stokes equations.
See Aoki--Maekawa \cite{Aok-Mae} for the corresponding result on a two-dimensional half space.
More recently, Hou--Wang--Yang \cite{2509.25116v1} announced an analogous statement in three dimensions with the unforced case $f\equiv 0$ by a computer-assisted proof.
\subsection{Main results}
The purpose of this paper is to show non-uniqueness of mild solutions to \eqref{eq:NS-intro} in $C([0,T);\dB_{p,q}^{n/p-1}(\R^n))$ for all $(p,q)$ satisfying neither (U1) nor (U2) in Proposition \ref{prop:uniq}.
Now, the precise statement of our main theorem reads as follows.
\begin{thm}\label{thm:main}
    Let $n \geq 3$.
    Let $p$ and $q$ satisfy either of the following.
    \begin{itemize}
        \item [\rm (N1)] $p=n$ and $2<q \leq \infty$,
        \item [\rm (N2)] $n < p \leq \infty$ and $1 \leq q \leq \infty$.
    \end{itemize}
    Then, there exist a monotone decreasing sequence $\{\varepsilon_m\}_{m \in \N} \subset (0,\infty)$ with $\varepsilon_m \to 0$ as $m \to \infty$ and a constant\footnote{We choose $\eta>0$ for $p<2n$ and $q<\infty$, while $\eta=0$ for the other cases.} $\eta=\eta(n,p,q)\geq 0$ such that for every $m \in \N$ and any initial datum $u_0 \in \dB_{p,q}^{n/p-1}(\R^n)$ with 
    \begin{align}
        \n{u_0}_{\dB_{p,q}^{n/p-1}} \leq \eta, \qquad \div u_0=0,
    \end{align}
    there exist a non-stationary global mild solution
    $u_m \in C([0,\infty);\dot{B}_{p,q}^{n/p-1}(\mathbb{R}^n))$
    to \eqref{eq:NS-intro} and a stationary solution $U_m \in \dB_{p,q}^{n/p-1}(\mathbb{R}^n)$ to \eqref{eq:sNS} satisfying
    \begin{align}\label{u_m}
        \lim_{t \to \infty} \n{ u_m(t) - U_m}_{\dB_{p,q}^{\frac{n}{p}-1}} = 0,
        \qquad
        C^{-1}\varepsilon_m \leq \n{U_m}_{\dB_{p,q}^{\frac{n}{p}-1}}\leq C\varepsilon_m
    \end{align}
    with some constant $C=C(n,p,q)>0$ independent of $u_0$ and $m$.
\end{thm}

\begin{rem}\label{rem:main}
Let us mention some remarks on our result.
\begin{enumerate}[label=\arabic*.]
    \item 
      Chemin \cite{Che-92} and Planchon \cite{Pla-98} established well-posedness for \eqref{eq:NS-intro} in critical Besov spaces
      for all $1\leq p<\infty$ and $1\leq q\leq \infty$.
      This does not contradict our result, since their solution space is 
      \begin{align}\label{Che-class}
      C([0,T);\dot{B}_{p,q}^{\frac{n}{p}-1}(\mathbb{R}^n))
      \cap 
      \widetilde{L^1}(0,T;\dot{B}_{p,q}^{\frac{n}{p}+1}(\mathbb{R}^n)),
      \end{align}
      whereas our non-unique solutions do not belong to the auxiliary dissipative space
      $\widetilde{L^1}(0,T;\dot{B}_{p,q}^{n/p+1}(\mathbb{R}^n))$.\footnote{See Section \ref{sec:pre} for the definition of Chemin--Lerner spaces $\widetilde{L^{\theta}}(0,T;\dot{B}_{p,q}^s(\R^n))$.}
    \item 
      The large-time behavior of global solutions to \eqref{eq:NS-intro} with small initial data in critical Besov spaces has been studied extensively; see \cites{Pla-98,Nakasato-25-JDE,Gallagher-Iftimie-Planchon-2002,Gallagher-Iftimie-Planchon-2003,Takeuchi-2025,Watanabe-2025,Benameu-2015-JMAA} for instance.
      However, all of these results exploit the viscous dissipation and address only solutions that decay and converge asymptotically to zero.
      In contrast, the non-decaying global solutions $u_{\varepsilon}$ constructed in this paper are the first examples\footnote{In the framework of weak solutions, the existence of non-dissipative solutions is widely known; see \cite{Buc-Vic-19} for instance.} of non-dissipative unforced Navier--Stokes flows in critical spaces in which the Stokes semigroup decays at temporal infinity.
    \item 
      For any $(p,q)$ satisfying (N1) or (N2), 
      the continuous embeddings
      \begin{align}
          \dot{B}_{n,2+\delta}^0(\mathbb{R}^n)\hookrightarrow \dot{F}_{n,2+\delta}^0(\mathbb{R}^n)\hookrightarrow \dot{F}_{p,q}^{\frac{n}{p}-1}(\mathbb{R}^n)
      \end{align}
      hold with some $\delta=\delta(p,q) \in (0,1)$.
      Thus, non-uniqueness for $u_0=0$ remains valid if one replaces the Besov spaces by the corresponding Lizorkin--Triebel spaces.
      In particular, we see that enlarging $L^n(\R^n)=\dot{F}_{n,2}^0(\R^n)$ in \cites{Lio-Mas,Fur-Rie-Ter-2000} even slightly to $\dot{F}_{n,q}^0(\R^n)$ with $q>2$
      leads to non-uniqueness.
      Moreover, we obtain the breakdown of unconditional uniqueness for small mild solutions in 
      $L^{\infty}(0,T;{\rm BMO}^{-1}(\R^n))$ since ${\rm BMO}^{-1}(\R^n)=\dot{F}_{\infty,2}^{-1}(\R^n)$.
      This is different from the phenomenon in \cite{Coi-Pal-25}, which concerns non-uniqueness of smooth solutions
      for large data in ${\rm BMO}^{-1}(\T^3)$.
      See Figure~\ref{fig:uniq-nonuniq} below.
    \item 
      On the results for the stationary Navier--Stokes equations in critical spaces,
      Chen \cite{Che-93} proved existence and uniqueness of small solutions in $L^n(\mathbb{R}^n)$,
      and Kaneko--Kozono--Shimizu \cite{Kan-Koz-Shi-19} established unique solvability in
      $\dot{B}_{p,q}^{n/p-1}(\mathbb{R}^n)$ for small external forces under condition (U1).
      Thus, for $(p,q)$ satisfying (U1) or (U2), small steady solutions are unique in $\dB_{p,q}^{n/p-1}(\R^n)$.
      Consequently, Theorem~\ref{thm:main} asserts sharp non-uniqueness of small solutions in $\dot{B}_{p,q}^{n/p-1}(\mathbb{R}^n)$.
      For the ill-posedness in the sense of discontinuity of the solution map, see \cites{Tsu-19-ARMA,Li-Yu-Zhu-25,Tan-Tsu-Zha-25,Fuj-24}.
    \item Concerning non-uniqueness for steady Navier--Stokes flows, Luo \cite{Luo-ARMA-2019}
      constructed nontrivial solutions to \eqref{eq:sNS} belonging to $L^2(\mathbb{T}^n)$ with $n \geq 4$.
      See \cite{Rieusset-25-JFA} for two-dimensional case in $H^{-1}(\T^2)$.
      Note that both spaces are supercritical and the three-dimensional case has been open.
      Our result appears to be the first to establish non-uniqueness of steady solutions in scaling critical spaces. Moreover, our result includes the three-dimensional case.
      Meanwhile, our method cannot be implemented in two dimensions. 
      This obstruction essentially reduces to the well-known difficulty of the solvability for the two-dimensional stationary Navier--Stokes equations on $\R^2$.  
    \item 
    As will be explained later, our construction of solutions relies on the structure provided by the classical Nash lemma, whereas we do not employ any of the periodic building blocks commonly used in convex integration approaches to fluid equations, such as Beltrami or Mikado flows. 
    \item 
    Theorem~\ref{thm:main} remains valid, with essentially the same proof, if $\mathbb{R}^n$ is replaced by $\mathbb{T}^n$.
    Shortly after the first version of the present paper appeared on arXiv, Cheskidov and Hou \cite{Che-Hou-26} obtained related periodic results.
    More precisely, they proved, for every $n \geq 2$, $1 \leq p,q \leq \infty$, and $s<0$, the existence of non-unique stationary solutions in $B_{p,q}^{s}(\mathbb{T}^n)$ and the failure of unconditional uniqueness of mild solutions in $C([0,T);B_{p,q}^{s}(\mathbb{T}^n))$.
    Since they stated that the condition $s<0$ includes part of the subcritical regime, their theorem is stated in a range wider than the critical scale emphasized here.
    We point out, however, that our construction in the torus case gives the full negative range $B_{p,q}^{s}(\mathbb{T}^n)$, $s<0$, $1 \leq p,q \leq \infty$.
    Moreover, at the case $s=0$, we still give non-uniqueness  whenever $2<q\leq \infty$, which is not covered by \cite{Che-Hou-26}.
\end{enumerate}
\end{rem}
\begin{figure}[H]
\centering
\begin{tikzpicture}[x=5.1cm,y=4cm,>=Stealth,line cap=round,line join=round]
  \def\yn{0.45}
  \draw[->] (0,0) -- (1.12,0) node[below] {$1/q$};
  \draw[->] (0,0) -- (0,1.12) node[left] {$1/p$};

  \node[below left] at (0,0) {$O$};
  \node[below] at (0.5,0) {$1/2$};
  \node[below] at (1,0) {$1$};
  \node[left] at (0,1) {$1$};
  \node[left] at (0,\yn) {$1/n$};

  \fill[blue!12] (0,\yn) rectangle (1,1);
  \fill[red!10] (0,0) rectangle (1,\yn);

  \draw[line width=2.6pt,red]  (0,0) -- (0,\yn);
  \draw[line width=2.2pt,blue] (0,\yn) -- (0,1) -- (1,1);
  \draw[very thick] (1,1) -- (1,0) -- (0,0);

  \draw[line width=2.6pt,red] (0,\yn) -- (0.5,\yn);
  \draw[line width=2.6pt,blue] (0.5,\yn) -- (1,\yn);

  \draw[line width=2.6pt,red] (0,0) -- (1,0);

  \draw[densely dashed] (0.5,0) -- (0.5,\yn);

  \fill[blue] (0.5,\yn) circle (0.020);  
  \fill[red]  (0.5,0) circle (0.015);      

  \draw[line width=2.6pt,red]  (1,0) -- (1,\yn);
  \draw[line width=2.6pt,blue] (1,\yn) -- (1,1);

  \draw[thin] (0.52,\yn+0.01) -- (0.62,\yn+0.09);
  \node[font=\tiny,anchor=west] at (0.62,\yn+0.09) {$L^n,\,\dot{B}^{0}_{n,2}$};

  \draw[thin] (0.52,0.01) -- (0.58,0.12);
  \node[font=\tiny,anchor=west] at (0.58,0.12) {$\mathrm{BMO}^{-1},\,\dot{B}^{-1}_{\infty,2}$};
\end{tikzpicture}
\caption{Let $X_{p,q}$ be either $\dot{B}_{p,q}^{n/p-1}(\R^n)$ or $\dot{F}_{p,q}^{n/p-1}(\R^n)$. The blue region indicates the range of $(1/q,1/p)$ for which mild solutions to~\eqref{eq:NS-intro} are unique in $C([0,T);X_{p,q})$ and small stationary solutions in $X_{p,q}$ are unique, whereas the red region indicates the range where uniqueness fails for both non-stationary and stationary solutions. Note that the continuous embedding $X_{p_1,q_1}\hookrightarrow X_{p_2,q_2}$ for $1\leq p_1\leq p_2\leq \infty$ and $1\leq q_1\leq q_2\leq \infty$ implies that the underlying function space becomes larger as a point in the figure moves leftward or downward.}
\label{fig:uniq-nonuniq}
\end{figure}
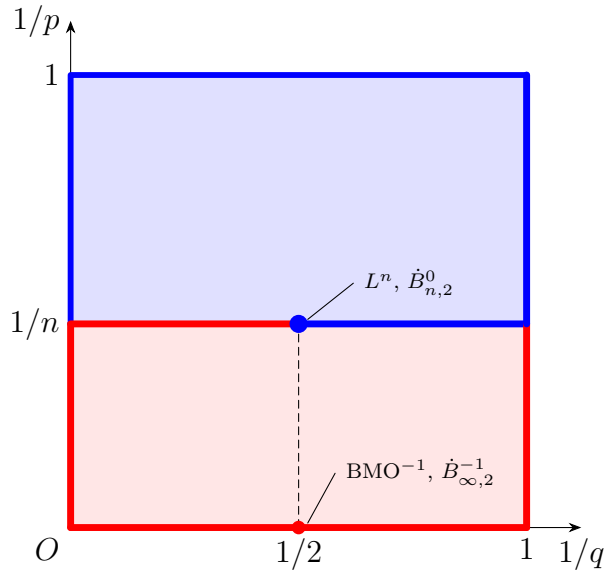
\raggedbottom
\subsection{Outline of the proof}\label{subsec:idea}
Let us explain the idea of the proof of Theorem \ref{thm:main}.
The strategy is that we first establish a dominant part of the solution which lies in $\dB_{p,q}^{n/p-1}(\R^n)$ with $(p,q)$ satisfying (N1) or (N2), and then solve the remaining perturbative problem by a fixed point argument
in a function space where well-posedness is available.
More precisely, we construct almost frequency-localized building blocks $\{V_j\}_{j\geq0}$ of stationary flows, each of which are essentially built on the
$j$th frequency shell, and design the principal profile
\begin{equation}\label{intro:V}
  V = \sum_{j=0}^{\infty} V_j \in \dot{B}_{p,q}^{\frac np-1}(\mathbb{R}^n)\setminus \{0\}
\end{equation}
so that the associated residual forcing
\begin{equation}\label{intro:defF}
  F := -\Delta V + \mathbb{P}\div (V\otimes V)
\end{equation}
belongs to well-posedness\footnote{See \cite{Kan-Koz-Shi-19} for well-posedness of the stationary Navier--Stokes equations in critical Besov spaces.} class $\dB_{r,1}^{n/r-3}(\R^n)$ ($r<n$) and admits the following smallness hierarchy:
\begin{equation}\label{intro:F-V}
  \|F\|_{\dot{B}_{r,1}^{\frac nr-3}} 
  \ll 
  \|V\|_{\dot{B}_{p,q}^{\frac np-1}} 
  \ll 
  1.
\end{equation}
Once \eqref{intro:F-V} is achieved, the perturbation\footnote{The perturbed system is given in \eqref{eq:W} below.} $W$ driven by $F$ can be constructed in
$\dot{B}_{r,1}^{n/r-1}(\mathbb{R}^n)$ by well-posedness theory, and we see that 
\begin{align}
    \n{W}_{\dB_{r,1}^{\frac{n}{r}-1}} \leq C\n{F}_{\dB_{r,1}^{\frac{n}{r}-3}} \ll \n{V}_{\dB_{p,q}^{\frac{n}{p}-1}}.
\end{align}
Hence, the desired steady state is obtained in the form $U = V + W$.
Note that $U \neq 0$ since 
\begin{align}
    \n{U}_{\dB_{p,q}^{\frac{n}{p}-1}} \geq \n{V}_{\dB_{p,q}^{\frac{n}{p}-1}} - C\n{W}_{\dB_{r,1}^{\frac{n}{r}-1}} \geq \frac{1}{2}\n{V}_{\dB_{p,q}^{\frac{n}{p}-1}} > 0.
\end{align}
Furthermore, to construct a non-stationary solution $u$ whose asymptotic profile is $U$, we decompose $u=U+w$.
Then, we may follow the standard stability argument to construct a small dissipative perturbed flow
$w \in C([0,\infty);\dot{B}_{p,q}^{n/p-1}(\mathbb{R}^n)) \cap \widetilde{L^\theta}(0,\infty;\dot{B}_{p,q}^{n/p - 1 + 2/\theta}(\mathbb{R}^n))$ with some $2<\theta<\infty$.
Then, the auxiliary dissipative norm enables us to show
$\|w(t)\|_{\dot{B}_{p,q}^{n/p-1}}\to0$ as $t\to\infty$.

We now explain how to construct the leading profile $V$ satisfying \eqref{intro:F-V}.
A natural first attempt, motivated by existing works such as~\cite{Coi-Pal-25}, would be to choose each $V_j$
as a superposition of plane waves concentrated on a single highly lacunary frequency scale
$N_j = A^{\gamma b^j}$ ($A \gg 1$), for instance,
$V_j(x)=\sum_{k\in\mathcal{K}} f_{k,j}(x)\cos(N_j a_k\cdot x)k$,
where $\mathcal{K}\subset\mathbb{R}^n$ is a finite set, $f_{k,j}$ are smooth amplitudes, and $a_k$ are constant vectors with $a_k \cdot k = 0$.
However, such a single-scale construction is too rigid to capture the delicate dependence on the Besov summability index $q$, which is essential for observing the subtle non-uniqueness phenomena addressed in (N1).
Instead, we distribute the $j$th building block over many frequency scales.
To this end, for $\lambda\gg1$ (which is a parameter corresponding to $\varepsilon_m^{-1/\alpha}$ with some $\alpha$) and a fixed large number $\mu \in \N$, we introduce a tetration-like sequence\footnote{This is defined by $\lambda_0=\lambda$ and $\lambda_j=\lambda^{2^{\lambda_{j-1}^{4\mu}}}$; see \eqref{eq:def-lambdaj}.}
$\{\lambda_j\}_{j\geq 0}$
and define the $j$th level frequency index set by
\[
  \Lambda_j \coloneqq \Bigl\{\ell\in\mathbb{N}\ ;\ \mu\log_{\lambda}\lambda_{j-1}\leq \ell \leq \log_{\lambda}\lambda_j\Bigr\},
\]
and take $V_j$ in the schematic form\footnote{In the actual construction of $V_j$, we apply a mollifier and introduce several additional modifications to enforce the divergence-free condition, so that its form becomes somewhat more involved.}
\begin{align}
    V_j(x)
    = 
    \sum_{k\in\mathcal{K}} 
    f_{k,j}(x)
    \sum_{\ell\in\Lambda_j}\frac{1}{\sqrt{h_j\ell}}\cos \sp{\lambda^{\ell} a_k\cdot x}k,
    \qquad
    h_j
    :=
    \sum_{\ell \in \Lambda_j}\frac{1}{\ell}.
\end{align}
Here, the choice of $1/\sqrt{\ell}$ is inspired by \cites{Yon-10,Tsu-19-ARMA}.
With this multi-scale design, 
the oscillation calculus in Section \ref{sec:osc} below yields 
\begin{align}
    \n{V_j}_{\dB_{p,q}^{\frac{n}{p}-1}}
    \sim 
    \sum_{k \in \mathcal{K}}
    \frac{\n{f_{k,j}}_{L^p}}{\sqrt{h_j}}
    \sp{
    \sum_{\ell \in \Lambda_j}
    \frac{\lambda^{(\frac{n}{p}-1)q\ell}}{\ell^{\frac{q}{2}}}
    }^{\frac{1}{q}},
\end{align}
which may shrink as $j \to \infty$ if and only if $(p,q)$ satisfies (N1) or (N2).
Thus, the series \eqref{intro:V} converges and we may define $V$.
For the analysis of $F$, we decompose $F$ as
\begin{align}
    F
  = \sum_{j=1}^{\infty}\Bigl(-\Delta V_{j-1}+\mathbb{P}\div(V_j\otimes V_j)\Bigr)
    + \sum_{j_1\neq j_2}\mathbb{P}\div(V_{j_1}\otimes V_{j_2})
    =:F_1+F_2.
\end{align}
The cross-shell interaction term $F_2$ is favorable, as it remains in the high-frequency region. By contrast, $F_1$ exhibits a  more delicate structure. Indeed, within the nonlinear interaction $\div(V_{j}\otimes V_j)$, resonances among modes inside the frequency shell $\Lambda_j$ generate an inverse-cascade transfer toward low frequencies, which exhibit the worst regularity. 
By selecting the directional set $\mathcal{K}$ and the amplitudes $f_{k,j}$ appropriately through the Nash geometric lemma, one may arrange for this low-frequency contribution to cancel the linear part $-\Delta V_{j-1}$. Consequently, the forcing term $F$ attains improved regularity and smaller magnitude relative to $V$.

\subsection{Organization of this paper}
This paper is organized as follows.
In Section \ref{sec:pre}, we summarize the notations used in this paper and elementary facts on Besov and Chemin--Lerner spaces.
In Section \ref{sec:osc}, we provide key Besov estimates for oscillating functions, which play a key role in constructing the leading part of the non-unique stationary solutions.
In Section \ref{sec:V}, we construct the principal part of the non-trivial steady state and calculate its Besov norms and the perturbative force.
In Section \ref{sec:pf}, we make use of the special flow constructed in Section \ref{sec:V} to complete the proof of Theorem \ref{thm:main}.

\section{Preliminaries}\label{sec:pre}
Throughout this paper, we denote by $C$ and $c$ the constants, which may differ in each line. In particular, $C=C(*,...,*)$ denotes the constant which depends only on the quantities appearing in parentheses. 
Let $\N$ be the set of all positive integers and put $\N_0:=\N \cup \{0\}$.
For $\mathbb{M} \subset \Z$, we define $2^{\mathbb{M}}:=\{2^m\ ;\ m \in \mathbb{M}\}$.
For two quantities $A_+$ and $A_-$, we use the notation
$\sum_{\pm} A_\pm := A_+ + A_-$.
The identity matrix is denoted by $\Id$.
For a metric space $(X,d)$, $x_0 \in X$, and $r>0$, set $B_X(x_0,r):=\{x \in X\ ;\ d(x,x_0) \leq r\}$.
Let $\mathscr{S}(\mathbb{R}^n)$ be the set of all Schwartz functions on $\mathbb{R}^n$, and 
$\mathscr{S}'(\mathbb{R}^n)$ denotes the set of all tempered distributions on $\mathbb{R}^n$.

Next, we recall the definitions of Besov and Chemin--Lerner spaces.
Let $\varphi_0 \in \mathscr{S}(\mathbb{R}^n)$ satisfy 
\begin{align}
    \supp \widehat{\varphi_0} \subset \Mp{ \xi \in \mathbb{R}^n\ ;\ 2^{-1} \leq | \xi | \leq 2 },\quad
    0 \leq \widehat{\varphi_0}(\xi) \leq 1, 
\end{align}
and 
\begin{align}
    \sum_{j \in \mathbb{Z}}
    \widehat{\varphi_j}(\xi) = 1 \qquad {\rm for\ all\ }\xi \in  \mathbb{R}^n \setminus \{0\},
\end{align}
where we have set $\varphi_j(x):=2^{nj}\varphi_0(2^jx)$.
Using this $\{\varphi_j\}_{j \in \Z}$, we may define the Littlewood--Paley decomposition as 
\begin{align}
    \Delta_j f := \mathscr{F}^{-1}\lp{\widehat{\varphi_j}(\xi)\widehat{f}(\xi)}
\end{align}
for $f \in \mathscr{S}'(\R^n)$ and $j \in \Z$.
For $1 \leq p,q \leq \infty$ and $s \in \mathbb{R}$, the Besov space $\dB_{p,q}^s(\mathbb{R}^n)$ is defined as 
\begin{align}
    \dB_{p,q}^s(\mathbb{R}^n)
    :={}&
    \Mp{
    f \in \mathscr{S}'(\mathbb{R}^n)/\mathscr{P}(\mathbb{R}^n)
    \ ; \ 
    \n{f}_{\dB_{p,q}^s}<\infty
    },\\
    \n{f}_{\dB_{p,q}^s}
    :={}&
    \n{
    \Mp{
    2^{sj}
    \n{\Delta_j f}_{L^p}
    }_{j \in \mathbb{Z}}
    }_{\ell^{q}(\mathbb{Z})},
\end{align}
where $\mathscr{P}(\mathbb{R}^n)$ is the set of all polynomials on $\mathbb{R}^n$.
It is well-known that if $s <n/p$ or $(s,q) = (n/p,1)$, then it holds
\begin{align}
    \dB_{p,q}^s(\mathbb{R}^n)
    \sim
    \Mp{
    f \in \mathscr{S}'(\mathbb{R}^n)
    \ ; \ 
    \n{f}_{\dB_{p,q}^s}<\infty,\quad
    f = \sum_{j \in \mathbb{Z}}\Delta_j f \quad {\rm in\ }\mathscr{S}'(\mathbb{R}^n)
    }.
\end{align}
Let us prepare a bilinear estimate in Besov spaces.
\begin{lemm}\label{lemm:nonlin-sta}
    Let $n \geq 3$.
    Then, for any $1 \leq p,q,r \leq \infty$ satisfying 
    \begin{align}\label{p-r}
        \frac{n}{p}+\frac{n}{r} > 2,
        \qquad
        0 \leq \frac{n}{r} - \frac{n}{p} < 1,
    \end{align}
    there exists a constant $C=C(n,p,q,r)>0$ such that
    \begin{align}
        \n{(-\Delta)^{-1}\mathbb{P}\div (V \otimes W)}_{\dB_{r,1}^{\frac{n}{r}-1}}
        \leq 
        C
        \n{V}_{\dB_{p,q}^{\frac{n}{p}-1}}
        \n{W}_{\dB_{r,1}^{\frac{n}{r}-1}}
    \end{align}
    for all $V \in \dB_{p,q}^{n/p-1}(\R^n)$ and $W \in \dB_{r,1}^{n/r-1}(\R^n)$.
\end{lemm}
We omit the precise proof of Lemma \ref{lemm:nonlin-sta} as it is immediately obtained by \cite{Abidi-Paicu-2007}*{Corollary 2.5}.
See \cites{Bah-Che-Dan-11,Saw-18} for more information on Besov spaces.

Next, we recall 
the Chemin--Lerner spaces which were first introduced in \cite{Che-Ler-95}:
\begin{align}
    \widetilde{L^{\theta}}(I;\dB_{p,q}^s(\R^n))
    :={}&
    \Mp{
    F:I \to  \mathscr{S}'(\mathbb{R}^n)/\mathscr{P}(\mathbb{R}^n)
    \ ; \ 
    \n{F}_{\widetilde{L^{\theta}}(I;\dB_{p,q}^s)}
    <\infty
    },\\
    \n{F}_{\widetilde{L^{\theta}}(I;\dB_{p,q}^s)}
    :={}&
    \n{
    \Mp{
    2^{sj}
    \n{\Delta_j F}_{L^\theta(I;L^p)}
    }_{j \in \mathbb{Z}}
    }_{\ell^{q}(\mathbb{Z})}
\end{align}
for $1 \leq p,q,\theta \leq \infty$, $s \in \mathbb{R}$, and an interval $I \subset \R$.
We also use the notation
\begin{align}
    \widetilde{C}(I ; \dB_{p,q}^s(\mathbb{R}^n))
    :=
    C(I ; \dB_{p,q}^s(\mathbb{R}^n))
    \cap
    \widetilde{L^{\infty}}(I ; \dB_{p,q}^s(\mathbb{R}^n)).
\end{align}
The Chemin--Lerner spaces are well-known to be particularly convenient when exploiting the maximal regularity of the heat kernel. 
We take advantage of this framework through the following lemma, which will be used in our stability analysis for non-stationary flows around the steady state.
\begin{lemm}[\cite{Fuj-AIPHC}]\label{lemm:nonlin-Duha}
Let $n \geq 2$.
Let $p$, $q$, $\theta$ satisfy
\begin{align}
    2 \leq p < 2n,
    \qquad
    1 \leq q \leq \infty,
    \qquad 
    2<\theta < \infty,
    \qquad
    \frac{n}{p}
    -1+\frac{1}{\theta}>0.
    \label{p-r-theta}
\end{align}
Then, there exists a constant $C=C(n,p,q,\theta)>0$ such that
\begin{align}
    &\n{
    \int_{t_0}^t
    e^{(t-\tau)\Delta}
    \mathbb{P}\div(v(\tau) \otimes w(\tau))d\tau
    }_{\widetilde{L^{\infty}}(I;\dB_{p,q}^{\frac{n}{p}-1}) \cap \widetilde{L^{\theta}}(I;\dB_{p,q}^{\frac{n}{p}-1+\frac{2}{\theta}})}\\
    &\quad\leq
    C
    \n{v}_{\widetilde{L^{\infty}}(I;\dB_{p,q}^{\frac{n}{p}-1})}
    \n{w}_{\widetilde{L^{\theta}}(I;\dB_{p,q}^{\frac{n}{p}-1+\frac{2}{\theta}})}
\end{align}
for all 
$I=(t_0,t_1) \subset \mathbb{R}$, 
$v \in \widetilde{L^{\infty}}(I;\dB_{p,q}^{n/p-1}(\mathbb{R}^n))$, 
and 
$w \in \widetilde{L^\theta}(I;\dB_{p,q}^{n/p-1+2/\theta}(\mathbb{R}^n))$.
\end{lemm}
\section{Highly oscillating functions in Besov spaces}\label{sec:osc}
In our analysis, we frequently use the sharp Besov estimates for finite superpositions of functions obtained by
modulating a smooth amplitude with a plane wave of the form \eqref{df:F} below.
If each Fourier transform $\widehat{f_k}$ of the amplitudes were compactly supported in a low frequency region, then the Fourier support of each $g_{\lambda,k,\ell}$ is concentrated around $\xi \sim \lambda^\ell a_k$, and we have 
\begin{align}
    \n{g_{\lambda,k,\ell}}_{\dB_{p,q}^s} \sim |b_\ell|\lambda^{s\ell}\n{f_k}_{L^p}
\end{align}
For sufficiently large $\lambda$,
these supports are disjoint,
which implies
\begin{equation}\label{eq:heuristic-superposition}
  \|F_{\lambda,L}\|_{\dot B^{s}_{p,q}}
  \sim
  \Mp{\sum_{\ell=1}^L \sp{\lambda^{s \ell}|b_\ell|}^q}^{\frac{1}{q}}
  \sum_{k=1}^K\n{f_k}_{L^p}.
\end{equation}
However, in our construction of the non-unique steady Navier--Stokes flow, 
the amplitudes $f_1,\dots,f_K$ necessarily belong to $C_c^\infty(\mathbb{R}^n)$, 
so their Fourier transforms decay rapidly but are not
compactly supported. 
Therefore, it is necessary to proceed with due care so that the Fourier tails of each $f_k$ may be treated as error terms. The following proposition realizes this.

\begin{prop}\label{prop:F-Besov}
    Let $1 \leq p,q \leq \infty$ and $s \in \R$ satisfy
    \begin{align}
        s>\frac{n}{p}-n.
    \end{align}
    Let $K,L \in \N$.
    Let $a_1,\dots,a_K \subset \mathbb{S}^{n-1}$ satisfy $a_k \neq a_{k'}$ unless $k=k'$.  
    For $\lambda \in 2^{\N}$, $b_1,\dots,b_L \in \C$, and $f_1,\dots,f_K \in \mathscr{S}(\R^n)$,
    set
    \begin{align}\label{df:F}
        F_{\lambda,L}(x)
        :=
        \sum_{k=1}^K
        \sum_{\ell=1}^L
        g_{\lambda,k,\ell}(x),
        \qquad
        g_{\lambda,k,\ell}(x):=b_{\ell}f_k(x)e^{i\lambda^\ell a_k\cdot x}.
    \end{align}
    Then, for any $M \in \N$ with $M > s$, there exists a constant $C>0$, depending only on $n$, $p$, $q$, $s$, $K$, $a_1,\dots,a_K$, and $M$ such that
    \begin{align}
        \n{F_{\lambda,L}}_{\dB_{p,q}^s}
        \leq{}
        &
        C\Mp{\sum_{\ell=1}^L\sp{\lambda^{s\ell}|b_{\ell}|}^q}^{\frac{1}{q}}
        \sum_{k=1}^K\n{f_k}_{L^p}
        \\
        &
        +
        C
        \sum_{k=1}^K
        \sum_{\ell=1}^L
        \lambda^{(s+n(1-\frac{1}{p})-M)\ell}
        |b_{\ell}|
        \n{\nabla^Mf_k}_{L^p\cap L^1}
        \label{est:upper-F}
    \end{align}
    and
    \begin{align}
        \n{F_{\lambda,L}}_{\dB_{p,q}^s}
        \geq{}
        &
        C^{-1}
        \Mp{\sum_{\ell=1}^L\sp{\lambda^{s\ell}|b_{\ell}|}^q}^{\frac{1}{q}}
        \max_{1 \leq k \leq K}\n{f_k}_{L^p}
        \\
        &
        -
        C
        \sum_{k=1}^K
        \sum_{\ell=1}^L
        \lambda^{(s-1)\ell}
        |b_{\ell}|
        \n{\nabla f_k}_{L^p}
        \\
        &
        -
        C
        \sum_{k=1}^K
        \sum_{\ell=1}^L
        \lambda^{(s+n(1-\frac{1}{p})-M)\ell}
        |b_{\ell}|
        \n{\nabla^Mf_k}_{L^p\cap L^1}
        \label{est:lower-F}
    \end{align}
    for all $L \in \N$ and $\lambda \in 2^{\N}$ with $\lambda\geq 2^{100}$.
\end{prop}
\begin{rem}
    For \eqref{est:upper-F} with $q=1$, we may remove the directional assumption on $a_1,\dots,a_K$.
    Indeed, applying \eqref{est:upper-F} with $K=L=1$ to each $g_{k,\ell}$ gives
    \begin{align}
        \n{F_{\lambda,L}}_{\dB_{p,1}^s}
        \leq{}&
        \sum_{k=1}^K
        \sum_{\ell=1}^L
        |b_{\ell}|
        \n{f_k(x)e^{i\lambda^{\ell}a_k \cdot x}}_{\dB_{p,1}^s}
        \\
        \leq{}&
        C
        \sum_{k=1}^K
        \sum_{\ell=1}^L
        \lambda^{s\ell}|b_{\ell}|
        \n{f_k}_{L^p}
        \\
        &
        +
        C
        \sum_{k=1}^K
        \sum_{\ell=1}^L
        \lambda^{(s+n(1-\frac{1}{p})-M)\ell}
        |b_{\ell}|
        \n{\nabla^M f_k}_{L^p\cap L^1}.
    \end{align}
\end{rem}
To prove Proposition \ref{prop:F-Besov}, 
we introduce operators that localize frequencies in the angular direction.
For any $a,a' \in \R^n \setminus \{0\}$, 
we define the angle between $a$ and $a'$ by
\begin{align}\label{df:theta}
    \vartheta(a,a')
    :=
    \operatorname{Arccos}\sp{\frac{a\cdot a'}{|a||a'|}}.
\end{align}
For $a_1,\dots,a_K$ given in Proposition \ref{prop:F-Besov},
we put  
\begin{align}
    \vartheta_*:=\min_{k \neq k'}\vartheta\sp{a_k,a_k'} \in (0,\pi].
\end{align}
Here, we set $\vartheta_*:=1$ if $K=1$.
Let $\widetilde{\phi}_1,\dots,\widetilde{\phi}_K \in C_c^{\infty}(\mathbb{S}^{n-1})$
satisfy 
$\supp \widetilde{\phi}_k \subset \{\sigma \in \mathbb{S}^{n-1}\ ;\ \vartheta(\sigma,a_k)<\vartheta_*/3\}$ 
and 
$\widetilde{\phi}_k(\sigma)=1$ for all $\sigma \in \mathbb{S}^{n-1}$ with $\vartheta(\sigma,a_k)\leq \vartheta_*/6$.
We also set $\widetilde{\phi}_0:=1-\sum_{k=1}^K\widetilde{\phi}_k$.
Using them, we define
\begin{align}
    P_kf:=
    \mathscr{F}^{-1}
    \lp{\phi_k(\xi)\widehat{f}(\xi)},
    \qquad
    \phi_k(\xi)
    :=
    \widetilde{\phi}_k\sp{\xi/|\xi|}
\end{align}
for $k=0,\dots,K$ and $f \in \mathscr{S}'(\R^n)$ with $0 \notin \supp \widehat{f}$.
We remark that there holds
\begin{align}
    P_m\Delta_jg_{\lambda,k,\ell}
    &
    =
    e^{i\lambda^\ell a_k\cdot x}
    \mathscr{F}^{-1}
    \lp{
    \widehat{\chi_{m,j}}(\xi+\lambda^{\ell}a_k)
    \widehat{f_k}(\xi)
    }
    \\
    &=
    e^{i\lambda^\ell a_k\cdot x}
    \sp{\chi_{m,j,k,\ell}*f_k}(x),
\end{align}
where we have defined 
\begin{align}
    \chi_{m,j}:=\mathscr{F}^{-1}[\phi_m\widehat{\varphi_j}],
    \qquad
    \chi_{m,j,k,\ell}(x):=e^{-i\lambda^\ell a_k\cdot x}\chi_{m,j}(x).
\end{align}
For these kernels, there holds
$\n{\chi_{m,j,k,\ell}}_{L^\rho}=\n{\chi_{m,j}}_{L^{\rho}}=2^{n(1-\frac{1}{\rho})j}\n{\chi_{m,0}}_{L^\rho}$ holds for $1 \leq \rho \leq \infty$.
Moreover, we prepare some notations;
for every $k \in\{1,\dots,K\}$ and $\ell \in \{1,\dots,L\}$,
set 
\begin{align}
    &
    j_{\ell}:=\ell\log_2\lambda,
    \\
    &
    R_{k,\ell}
    :=
    \{(m,j) \in \{0,\dots,K\} \times \Z\ ;\ (m,j)\neq(k,j_{\ell}),(k,j_{\ell}\pm 1)\}.
\end{align}

For the proof of Proposition \ref{prop:F-Besov},
we prepare a lemma.
\begin{lemm}\label{lemm:chi}
    There exists a constant $c_*=c_*(K,a_1,\dots,a_K) \in 2^{\Z}$ such that
    for any $\ell \in \{1,\dots,L\}$, $k \in\{1,\dots,K\}$, $(m,j) \in R_{k,\ell}$, and $\lambda \in 2^{\N}$,
    there holds 
    \begin{align}
        \supp \mathscr{F}\lp{\chi_{m,j,k,\ell}}
        \subset
        \Mp{
        \xi \in \mathbb{R}^n\ ;\ |\xi| \geq c_* \max\{2^j,\lambda^{\ell}\}
        }.
    \end{align}
\end{lemm}
\begin{proof}
Let $\xi \in \R^n$ satisfy $\mathscr{F}\lp{\chi_{m,j,k,\ell}}(\xi)=\widehat{\chi_{m,j}}(\xi+\lambda^\ell a_k) \neq 0$.
We first consider the case of $m=k$.
Then, by $2^{j-1} \leq |\xi+\lambda^{\ell}a_k| \leq 2^{j+1}$ and $\lambda^{\ell}=2^{j_{\ell}}$, it holds
\begin{align}
    |\xi|
    &=
    |\xi+\lambda^\ell a_k - \lambda^\ell a_k|
    \\
    &\geq 
    \begin{cases}
        |\xi+\lambda^{\ell}a_k| - \lambda^{\ell}
        \geq 
        2^{j-1}-2^{j-2}
        &
        (j \geq j_{\ell}+2),
        \\
        \lambda^{\ell}-|\xi+\lambda^{\ell}a_k|
        \geq 
        2^{j_{\ell}}-2^{j_{\ell}-1}
        &
        (j \leq j_{\ell}-2),
    \end{cases}
    \\
    &\geq
    c
    \max\Mp{2^j,\lambda^{\ell}}.
\end{align}
For the case of $m\neq k$, we see that 
\begin{align}
    |\xi|^2
    ={}&
    |\xi+\lambda^{\ell}a_k|^2+\lambda^{2\ell}
    -
    2\lambda^{\ell}|\xi+\lambda^{\ell}a_k|\cos \vartheta(\xi+\lambda^{\ell}a_k,\lambda^{\ell}a_k)
    \\
    ={}&
    \frac{1}{2}\sp{|\xi+\lambda^{\ell}a_k|-\lambda^{\ell}\cos \vartheta(\xi+\lambda^{\ell}a_k,\lambda^{\ell}a_k)}^2
    \\
    &
    +
    \frac{1}{2}\sp{\lambda^{\ell}-|\xi+\lambda^{\ell}a_k|\cos \vartheta(\xi+\lambda^{\ell}a_k,\lambda^{\ell}a_k)}^2
    \\
    &
    +
    \frac{1}{2}
    \sin^2\vartheta(\xi+\lambda^{\ell}a_k,\lambda^{\ell}a_k)
    \sp{\lambda^{2\ell}+|\xi+\lambda^{\ell}a_k|^2}
    \\
    \geq{}&
    \frac{1}{2}
    \sin^2(\vartheta_*/6)
    \sp{\lambda^{2\ell}+|\xi+\lambda^{\ell}a_k|^2},
\end{align}
which yields $|\xi|\geq c\max\Mp{2^j,\lambda^{\ell}}$.
Thus, we complete the proof.
\end{proof}
Let us now provide the proof of Proposition \ref{prop:F-Besov}
\begin{proof}[Proof of Proposition \ref{prop:F-Besov}]
We split the proof into two steps.
%

\noindent
{\it Step 1. The upper bound estimate.}
%
We may decompose $F_{\lambda,L}$ as 
\begin{align}
    F_{\lambda,L}
    ={}
    \sum_{k=1}^K
    \sum_{\ell=1}^L
    \sum_{m=0}^K
    \sum_{j \in \Z}
    P_m\Delta_jg_{\lambda,k,\ell}
    ={}
    F_{\lambda,L}^{(\p)}
    +
    F_{\lambda,L}^{(\r)},
\end{align}
where the principal part $F_{\lambda,L}^{(\p)}$ and the remainder part $F_{\lambda,L}^{(\r)}$ are defined by
\begin{align}
    &
    F_{\lambda,L}^{(\p)}
    :=
    \sum_{k=1}^K
    \sum_{\ell=1}^L
    \sum_{|j-j_{\ell}|\leq 1}
    P_{k}\Delta_jg_{\lambda,k,\ell},
    \\
    &
    \begin{aligned}
    F_{\lambda,L}^{(\r)}
    :={}&
    \sum_{m=0}^K
    \sum_{k=1}^K
    \sum_{\ell=1}^L
    \sum_{(m,j) \in R_{k,\ell}}
    P_m\Delta_jg_{\lambda,k,\ell}.
    \end{aligned}
\end{align}
For the estimate of the principal part, we see that 
\begin{align}
    \n{F_{\lambda,L}^{(\p)}}_{\dB_{p,q}^s}
    &
    \leq 
    \sum_{k=1}^K
    \Mp{
    \sum_{j' \in \Z}
    \sp{
    2^{sj'}
    \n{\sum_{\ell=1}^L
    \sum_{|j-j_{\ell}|\leq 1}
    P_{k}\Delta_j\Delta_{j'}g_{\lambda,k,\ell}
    }_{L^p}
    }^q
    }^{\frac{1}{q}}
    \\
    &
    =
    \sum_{k=1}^K
    \Mp{
    \sum_{\ell=1}^L
    \sum_{|j'-j_{\ell}|\leq 2}
    \sp{
    2^{sj}
    \n{
    \sum_{|j-j_{\ell}|\leq 1}
    P_k\Delta_j\Delta_{j'}g_{\lambda,k,\ell}
    }_{L^p}
    }^q
    }^{\frac{1}{q}}
    \\
    &
    \leq 
    C
    \sum_{k=1}^K
    \Mp{
    \sum_{\ell=1}^L
    \sum_{|j'-j_{\ell}|\leq 2}
    \sp{
    2^{sj}
    \n{
    g_{\lambda,k,\ell}
    }_{L^p}
    }^q
    }^{\frac{1}{q}}
    \\
    &
    =
    C
    \sum_{k=1}^K
    \Mp{
    \sum_{\ell=1}^L
    \sum_{|j'-j_{\ell}|\leq 2}
    \sp{
    2^{sj}
    |b_\ell|
    }^q
    }^{\frac{1}{q}}
    \n{f_k}_{L^p}
    \\
    &
    \leq
    C
    \Mp{
    \sum_{\ell=1}^L
    \sp{
    \lambda^{s\ell}
    |b_\ell|
    }^q
    }^{\frac{1}{q}}
    \sum_{k=1}^K
    \n{f_k}_{L^p}.
    \label{est:up-Fp}
\end{align}
For the remainder terms, there holds
\begin{align}
    \n{F_{\lambda,L}^{(\r)}}_{\dB_{p,1}^s}
    &
    \leq 
    \sum_{k=1}^K
    \sum_{\ell=1}^L
    \sum_{(m,j) \in R_{k,\ell}}
    2^{sj'}
    \n{P_{k'}\Delta_j\Delta_{j'}g_{\lambda,k,\ell}}_{L^p}
    \\
    &
    \leq 
    C
    \sum_{k=1}^K
    \sum_{\ell=1}^L
    \sum_{(k',j) \in R_{k,\ell}}
    2^{sj}
    \n{P_{k'}\Delta_jg_{\lambda,k,\ell}}_{L^p}.
\end{align}
Here, Lemma \ref{lemm:chi} yields 
\begin{align}
    \n{P_m\Delta_jg_{\lambda,k,\ell}}_{L^p}
    &
    \leq 
    |b_\ell|
    \sum_{j'\geq \max\{j,j_{\ell}\}+j_*}
    \n{
    \Delta_{j'}
    \chi_{m,j,k,\ell}*f_k
    }_{L^p}
    \\
    &\leq 
    C|b_\ell|
    \sum_{j'\geq \max\{j,j_{\ell}\}+j_*}
    2^{-Mj'}
    \n{
    \chi_{m,j,k,\ell}*\nabla^Mf_k
    }_{L^p}
    \\
    &=
    C|b_\ell|\sp{\max\{2^j,\lambda^\ell\}}^{-M}
    \n{\chi_{m,j,k,\ell}*\nabla^Mf_k}_{L^p},
\end{align}
where we have put $j_* = \log_2(2c_*)$. 
Thus, we obtain
\begin{align}
    \n{F_{\lambda,L}^{(\r)}}_{\dB_{p,1}^s}
    \leq {}&
    C
    \sum_{k=1}^K
    \sum_{\ell=1}^L
    \sum_{(m,j) \in R_{k,\ell}}
    |b_\ell|
    2^{sj}
    \sp{\max\{2^j,\lambda^\ell\}}^{-M}
    \n{\chi_{m,j,k,\ell}*\nabla^Mf_k}_{L^p}\\
    \leq {}&
    C
    \sum_{k=1}^K
    \sum_{\ell=1}^L
    \sum_{m=0}^K
    \sum_{j \geq j_{\ell}+1}
    |b_\ell|
    2^{(s-M)j}
    \n{\chi_{m,j,k,\ell}}_{L^1}
    \n{\nabla^Mf_k}_{L^p}
    \\
    & + 
    C
    \sum_{k=1}^K
    \sum_{\ell=1}^L
    \sum_{m=0}^K
    \sum_{j \leq j_{\ell}}
    |b_\ell|
    2^{sj}
    (\lambda^{-M\ell}
    \n{\chi_{m,j,k,\ell}}_{L^p}
    \n{\nabla^Mf_k}_{L^1}\\
    \leq {}&
    C
    \sum_{k=1}^K
    \sum_{\ell=1}^L
    \sum_{j \geq j_{\ell}+1}
    |b_\ell|
    2^{(s-M)j}
    \n{\nabla^Mf_k}_{L^p}
    \\
    & + 
    C
    \sum_{k=1}^K
    \sum_{\ell=1}^L
    \sum_{j \leq j_{\ell}}
    |b_\ell|
    2^{\{s+n(1-\frac{1}{p})\}j}
    \lambda^{-M\ell}
    \n{\nabla^Mf_k}_{L^1}\\
    \leq {}&
    C
    \sum_{k=1}^K
    \sum_{\ell=1}^L
    \lambda^{(s-M)\ell}
    |b_{\ell}|
    \n{\nabla^Mf_k}_{L^p}
    \\
    &
    +
    C
    \sum_{k=1}^K
    \sum_{\ell=1}^L
    \lambda^{(s+n(1-\frac{1}{p})-M)\ell}
    |b_{\ell}|
    \n{\nabla^Mf_k}_{L^1}.
    \label{est:up-Fr}
\end{align}
Combining \eqref{est:up-Fp} and \eqref{est:up-Fr}, we obtain the desired estimate.
%

\noindent
{\it Step 2. The lower bound estimate.}
%
Since $P_0,P_1,\dots,P_K$ are bounded in homogeneous Besov spaces, there holds
\begin{align}
    \n{F_{\lambda,L}}_{\dB_{p,q}^s}
    \geq 
    c
    \max_{1 \leq m \leq K}
    \n{P_mF_{\lambda,L}}_{\dB_{p,q}^s}.
\end{align}
Fix a $m \in \{1,\dots,K\}$ and
decompose $P_mF_{\lambda,L}$ as
\begin{align}
    &P_mF_{\lambda,L}
    =
    F^{(\p)}_{\lambda,m,L}
    +
    F^{(\r)}_{\lambda,m,L};
    \\
    &
    \quad
    F^{(\p)}_{\lambda,m,L}
    :=
    \sum_{\ell=1}^L
    \sum_{|j-j_{\ell}|\leq 1}
    P_m\Delta_jg_{\lambda,m,\ell},
    \\
    &
    \quad
    F^{(\r)}_{\lambda,m,L}
    :=
    \sum_{k=1}^K
    \sum_{\ell=1}^L
    \sum_{j:(m,j) \in R_{k,\ell}}
    P_m\Delta_jg_{\lambda,k,\ell}.
\end{align}
For the lower bound estimate of the principal part $F^{(\p)}_{\lambda,m,L}$, we see that 
\begin{align}
    \n{F_{\lambda,m,L}^{(\p)}}_{\dB_{p,q}^s}
    &
    =
    \Mp{
    \sum_{j' \in \mathbb{Z}}
    \sp{
    2^{sj'}
    \n{
    \sum_{\ell=1}^L
    \sum_{|j-j_{\ell}|\leq 1}
    P_m\Delta_j\Delta_{j'}g_{\lambda,m,\ell}
    }_{L^p}
    }^q
    }^{\frac{1}{q}}
    \\
    &
    =
    \Mp{
    \sum_{\ell=1}^L
    \sum_{|j'-j_{\ell}|\leq 2}
    \sp{
    2^{sj'}
    \n{
    \sum_{|j-j_{\ell}|\leq 1}
    P_m\Delta_j\Delta_{j'}g_{\lambda,m,\ell}
    }_{L^p}
    }^q
    }^{\frac{1}{q}}
    \\
    &
    \geq 
    \Mp{
    \sum_{\ell=1}^L
    \sp{
    2^{sj_{\ell}}
    \n{
    \sum_{|j-j_{\ell}|\leq 1}
    P_m\Delta_j\Delta_{j_{\ell}}g_{\lambda,m,\ell}
    }_{L^p}
    }^q
    }^{\frac{1}{q}}
    \\
    &
    =
    \Mp{
    \sum_{\ell=1}^L
    \sp{
    2^{sj_{\ell}}
    \n{
    P_m\Delta_{j_{\ell}}g_{\lambda,m,\ell}
    }_{L^p}
    }^q
    }^{\frac{1}{q}}.
    \label{est:low-Fp}
\end{align}
Here, there holds
\begin{align}
    P_m\Delta_{j_{\ell}}g_{\lambda,m,\ell}(x)
    &=
    b_{\ell}
    \int_{\R^n}
    \chi_{m,j_{\ell}}(y)f_m(x-y)e^{i\lambda^\ell a_m\cdot(x-y)}dy
    \\
    &
    =
    b_{\ell}
    e^{i\lambda^\ell a_m\cdot x}
    f_m(x)
    \int_{\R^n}
    \chi_{m,j_{\ell}}(y)e^{-i\lambda^\ell a_m\cdot y}
    dy\\
    &\qquad
    +
    b_{\ell}
    e^{i\lambda^\ell a_m\cdot x}
    \int_{\R^n}
    (f_m(x-y)-f_m(x))\chi_{m,j_{\ell}}(y)e^{-i\lambda^\ell a_m\cdot y}
    dy\\
    &
    =:
    b_{\ell}
    e^{i\lambda^\ell a_m\cdot x}
    f_m(x)
    +
    R_{\lambda,m,\ell}(x),
\end{align}
where we have used $\widehat{\chi_{m,j_{\ell}}}(\lambda^\ell a_m)=1$.
The remainder term $R_{\lambda,m,\ell}$ is bounded as
\begin{align}
    \n{R_{\lambda,m,\ell}}_{L^p}
    &\leq 
    |b_{\ell}|\n{\nabla f_m}_{L^p}
    \int_{\R^n} |y||\chi_{m,j_{\ell}}(y)|dy
    \\
    &\leq 
    C|b_{\ell}|2^{-j_{\ell}}\n{\nabla f_m}_{L^p}
    \\
    &= 
    C|b_{\ell}|\lambda^{-\ell}\n{\nabla f_m}_{L^p}.
\end{align}
Thus, we see that 
\begin{align}\label{est:low-g}
    \n{P_m\Delta_{j_{\ell}}g_{\lambda,m,\ell}}_{L^p}
    &
    \geq 
    |b_{\ell}|
    \n{f_m}_{L^p}
    -
    C\lambda^{-\ell}
    |b_{\ell}|
    \n{\nabla f_m}_{L^p}.
\end{align}
By \eqref{est:low-Fp} and \eqref{est:low-g}, it holds
\begin{align}
    \n{F_{\lambda,m,L}^{(\p)}}_{\dB_{p,q}^s}
    \geq{}&
    c
    \Mp{\sum_{\ell=1}^L\sp{\lambda^{s\ell}|b_{\ell}|}^q}^{\frac{1}{q}}
    \n{f_m}_{L^p}
    -
    C
    \sum_{\ell=1}^L
    \lambda^{(s-1)\ell}
    |b_{\ell}|
    \n{\nabla f_m}_{L^p}.
    \label{est:lw-Fp}
\end{align}
On the estimate of $F^{(\r)}_{\lambda,m,L}$,
the similar strategy as for the estimate of $F^{(\r)}_{\lambda,m,L}$ yields
\begin{align}
    \n{F^{(\r)}_{\lambda,m,L}}_{\dB_{p,1}^s}
    \leq{}& 
    C
    \sum_{k=1}^K
    \sum_{\ell=1}^L
    \sum_{j:(m,j) \in R_{k,\ell}}
    \sum_{|j'-j|\leq1}
    2^{sj'}
    \n{P_m\Delta_{j'}\Delta_jg_{\lambda,k,\ell}}_{L^p}\\
    \leq{}& 
    C
    \sum_{k=1}^K
    \sum_{\ell=1}^L
    \sum_{j:(m,j) \in R_{k,\ell}}
    2^{sj}
    \n{P_m\Delta_jg_{\lambda,k,\ell}}_{L^p}\\
    \leq {}&
    C
    \sum_{k=1}^K
    \sum_{\ell=1}^L
    \lambda^{(s-M)\ell}
    |b_{\ell}|
    \n{\nabla^Mf_{k}}_{L^p}
    \\
    &
    +
    C
    \sum_{k=1}^K
    \sum_{\ell=1}^L
    \lambda^{(s+n(1-\frac{1}{p})-M)\ell}
    |b_{\ell}|
    \n{\nabla^Mf_{k}}_{L^p \cap L^1}.
    \label{est:lw-Fr}
\end{align}
Gathering \eqref{est:lw-Fp} and \eqref{est:lw-Fr}, we obtain \eqref{est:lower-F} and complete the proof.
\end{proof}

\section{Principal part of solutions and perturbed forces}\label{sec:V}
In this section, we construct the principal part of the non-unique steady flow outlined in Section~\ref{subsec:idea}, and prove that the perturbative forcing term admits better regularity and smallness than the principal part. 
The central aim of this section is to show the following theorem.
\begin{thm}\label{thm:p-sol}
    For $n \geq 3$, 
    there exists a constant $\lambda_{*0}>0$ such that for every $\lambda \in 2^{\N}$ with $\lambda \geq \lambda_{*0}$, 
    there exists a divergence-free vector field $V_\lambda \in \mathscr{S}'(\R^n)$
    such that the following {\rm (i)} and {\rm (ii)} hold true.
    \begin{enumerate}
        \item [\rm (i)]
        For any $(p,q)$ satisfying {\rm (N1)} or {\rm (N2)} in Theorem \ref{thm:main},
        $V_\lambda$ belongs to $\dB_{p,q}^{n/p-1}(\R^n)$ and it holds 
        \begin{align}\label{est:V-optimal}
            C^{-1}\lambda^{-\alpha(p)}
            \leq 
            \n{V_{\lambda}}_{\dB_{p,q}^{\frac{n}{p}-1}} 
            \leq 
            C\lambda^{-\alpha(p)}
        \end{align}
        with $\alpha(p):=(3-n/p)(4n+5)-4-n/p$ and some constant $C=C(n,p,q)>0$.
        \item [\rm (ii)] 
        For any $n/2< r<n$,
        the perturbative force 
        \begin{align}
            F_{\lambda}:=-\Delta V_{\lambda} + \mathbb{P}\div (V_{\lambda} \otimes V_{\lambda})
        \end{align}
        belongs to $\dB_{r,1}^{n/r-3}(\R^n)$ and satisfies
        \begin{align}\label{est:F}
            \n{F_{\lambda}}_{\dB_{r,1}^{\frac{n}{r}-3}}
            \leq C\lambda^{-\beta}
        \end{align}
        with $\beta:=12n+12$ and some constant $C=C(n,r)>0$.
    \end{enumerate}
\end{thm}
\begin{rem}
    Since $\alpha(p) \leq \alpha(\infty) = 12 n+11 < \beta$ and $\dB_{r,1}^{n/r-1}(\R^n) \hookrightarrow \dB_{p,q}^{n/p-1}(\R^n)$, 
    $F_{\lambda}$ is much smaller and more regular than $V_{\lambda}$.
\end{rem}
\subsection{Notations and key lemmas}
Before proving Theorem~\ref{thm:p-sol}, we collect several auxiliary results and notations.
We begin by recalling the well-known geometric lemma.
\begin{lemm}[\cite{Nash-54}]\label{lemm:Nash}
Let $n \geq 2$.
There exist a positive constant $r_0$, a bounded closed interval $\mathcal{I} \subset (0,\infty)$, a finite subset $\mathcal{K}$ of $\Z^n \setminus \{0\}$, and $C^{\infty}$ functions $\Gamma_k:B_{\mathcal{S}_n}(\Id,r_0) \to \mathcal{I}$ $(k \in \mathcal{K})$ 
such that
\begin{align}\label{Nash}
    M = \sum_{k \in \mathcal{K}} \Gamma_k(M)^2 k \otimes k,
    \qquad
    M \in B_{\mathcal{S}_n}(\Id,r_0).
\end{align}
Here, $\mathcal{S}_n$ stands for the set of all real symmetric $n \times n$ matrices.
\end{lemm}
\begin{proof}
    For the readers' convenience, we provide an outline of an elementary proof.
    Let us define $\mathcal{K}$ by 
    \begin{align}
        \mathcal{K}
        :=
        \Mp{
        \bb{e}_i\ ;\ i=1,\dots,n
        }
        &
        \cup 
        \Mp{
        \bb{e}_i+\bb{e}_j
        \ ;\ i,j=1,\dots,n,\ i < j
        }
        \\
        &
        \cup 
        \Mp{
        \bb{e}_i-\bb{e}_j
        \ ;\ i,j=1,\dots,n,\ i < j
        },
    \end{align}
    where $\bb{e}_1,\dots,\bb{e}_n \in \R^n$ stands for the $i$-th standard basis vector.
    We define $\{\Gamma_k\}_{k \in \mathcal{K}}$ by
    \begin{align}
        \Gamma_k(M)
        :=
        \begin{cases}
            \sp{M_{ii}-\dfrac{1}{2}}^{\frac{1}{2}}, & (k=\bb{e}_i,\ i=1,\dots,n),
            \\[7pt]
            \sp{\dfrac{1}{4(n-1)}
            \pm 
            \dfrac{M_{ij}}{2}
            }^{\frac{1}{2}},  & (k=\bb{e}_i\pm \bb{e}_j,\ i,j=1,\dots,n,\ i < j),
        \end{cases}
    \end{align}
    for all $M=(M_{ij}) \in \mathcal{S}_n$ with $|M-\Id| \leq r_0:=(4(n-1))^{-1}$.
    Then, we may easily check that \eqref{Nash} holds true, and complete the proof.
\end{proof}
Here, we introduce some notations that are to be used throughout this section.
\begin{enumerate}[label=\arabic*.]
\item 
Let $\{ a_k \}_{k \in \mathcal{K}} \subset \mathbb{S}^{n-1}$ satisfy that $a_k$ and $a_{k'}$ are not parallel unless $k = k'$.
For instance, we may choose
\begin{align}
    a_k
    :=
    \begin{cases}
        \bb{e}_{i+1}, & (k=\bb{e}_i,\ i=1,\dots,n-1), \\
        \bb{e}_1, & (k=\bb{e}_n), \\
        \dfrac{\bb{e}_i \mp \bb{e}_j}{\sqrt{2}}, & (k=\bb{e}_i \pm \bb{e}_j,\ i,j=1,\dots,n,\ i < j).
    \end{cases}
\end{align}
\item 
Let $\mu=4n+5$.
For $\lambda\in 2^{\mathbb{N}}$, we define $\{\lambda_j\}_{j\in\N_0}$ inductively by
\begin{equation}\label{eq:def-lambdaj}
  \lambda_0=\lambda,
  \qquad
  \lambda_j=\exp_{\lambda}\!\left(\exp_2\!\left(\lambda_{j-1}^{4\mu}\right)\right),
\end{equation}
where we have used the notation $\exp_a(b)\coloneqq a^b$ for $a>0$ and $b\in\mathbb{R}$.

Throughout this section, $\lambda$ may be replaced by a larger dyadic number whenever necessary, without further comment. We emphasize that the required lower bound on $\lambda$ is chosen independently of $j$. 
Moreover, almost all quantities introduced below depend on the parameter $\lambda$; however, we do not record this dependence explicitly in their notations so as to avoid unnecessarily complicating them.
We also note that the generic constants are always independent of $\lambda$ and $j$.
\item 
For $j \in \N$, set 
\begin{align}
    \Lambda_j:=
    \Mp{\ell \in \N\ ;\ \mu \log_{\lambda}\lambda_{j-1} \leq \ell \leq \log_{\lambda} \lambda_j}.
\end{align}
In what follows, we choose $\lambda$ sufficiently large so that $\Lambda_j \neq \varnothing$ for all $j \in \N$.
\item 
We define $\{h_j\}_{j \in \N}$ by
\begin{align}
    h_j:=
    \sum_{\ell \in \Lambda_j}
    \frac{1}{\ell}.
\end{align}
Note that there holds
\begin{align}
    h_j 
    \sim{}& 
    \log_2 
    \frac{\log_\lambda\lambda_j}{\mu\log_\lambda\lambda_{j-1}}
    \\
    ={}&
    \begin{cases}
        \lambda_{j-1}^{4\mu}-\log_2\mu & (j=1), \\
        \lambda_{j-1}^{4\mu}-\lambda_{j-2}^{4\mu}-\log_2\mu, \quad & (j =2,3,4,\dots),
    \end{cases}
\end{align}
which yields $h_j \sim \lambda_{j-1}^{4\mu}$.
\item 
We choose a function $\rho \in C_c^{\infty}(\mathbb{R}^n;[0,\infty))$ so that
\begin{align}
    \supp \rho = B_{\R^n}(0,1), 
    \qquad
    \int_{\R^n} \rho(y) dy = 1.
\end{align}
Let $\rho_{j}$ be the mollifier of the scale $\lambda_j^{-2}$; more precisely, we define
\begin{align}
    \rho_{j}(x):=\lambda_j^{2n}\rho\sp{\lambda_j^2x}.
\end{align}
\item 
Using a function $\psi \in C^{\infty}(\R)$ satisfying $\bb{1}_{(-\infty,0]} \leq \psi \leq \bb{1}_{(-\infty,1)}$, we define a sequence $\{ \psi_{j} \}_{j \in \N_0} \subset C_c^{\infty}(\R^n)$ by
\begin{align}
    \psi_{j}(x)
    :=
    \psi(|x|-\lambda^j).
\end{align}
\item 
Let us define the differential operator $\mathcal{D}$ by
\begin{align}
    \mathcal{D}v:=-\nabla v - (\nabla v)^{\top} + 2 (\div v) \Id 
\end{align}
for all smooth vector field $v:\R^n \to \R^n$.
Note that $\mathcal{D}v$ is a real symmetric matrix and it holds
\begin{align}
    \div \mathcal{D}v = -\Delta v + \nabla \div v = -\Delta \mathbb{P} v.
\end{align}
\end{enumerate}
In order to define the building blocks of the desired $V_{\lambda}$, we introduce the following operators.
\begin{df}
    For $j \in \N$ and $v \in C_c^{\infty}(\R^n)$ satisfying $\n{\mathcal{D}v}_{L^\infty} \leq r_0 \lambda_{j-1}^{8}$, we define 
    \begin{align}
        &
        {\mathcal{V}}_j[v]
        :=
        \sqrt{2}
        \lambda_{j-1}^4
        \rho_j*
        \sum_{k \in \mathcal{K}}
        \div \mathcal{D}
        \sp{{\Gamma}_{j,k}[v]
        {\Psi}_{j,k}
        k},
    \end{align}
    where we have set 
    \begin{align}
        &{\Gamma}_{j,k}[v](x):=\Gamma_k\sp{\Id-\lambda_{j-1}^{-8}\mathcal{D}v(x)},
        \\
        &
        {\Psi}_{j,k}(x)
        :=
        \sum_{\ell\in \Lambda_j}
        \frac{1}{\lambda^{2\ell}\sqrt{h_j\ell}}
        \psi_{j}(x)
        \cos \sp{\lambda^{\ell}
        a_k \cdot x}.
    \end{align}
\end{df}
\begin{rem}\label{rem:V}
    When $\R^n$ is replaced by $\T^n$, 
    it is not necessary to multiply by $\psi_j$ in the definition of $\Psi_{j,k}$.
\end{rem}
Let us establish elementary derivative estimates for each $\mathcal{V}_j[v]$, which rely on the mollifier.
Although these bounds are far from optimal, they ensure the well-definedness of Definition~\ref{df:V_j} below,
and they provide an essential foundation for obtaining the sharp estimates later on.
\begin{lemm}\label{lemm:V[v]}
    For every $m \in \N_0$, 
    there exists a positive constant $C=C(m)$ such that 
    \begin{align}\label{est:V[v]-1}
        \n{\nabla^m{\mathcal{V}}_j[v]}_{L^{p}}
        \leq 
        C
        \lambda_j^{2m+4}.
    \end{align}
    for all $j \in \N$,  $1 \leq p \leq \infty$, and $v \in C_c^{\infty}(\R^n)$ with $\n{\mathcal{D}v}_{L^\infty} \leq r_0 \lambda_{j-1}^{8}$.
    Moreover, it holds
    \begin{align}\label{est:V[v]-2}
        \n{\mathcal{D}{\mathcal{V}}_j[v]}_{L^{\infty}}
        \leq 
        r_0\lambda_{j}^{8}.
    \end{align}
\end{lemm}
\begin{proof}
    From the Young inequality, we see that 
    \begin{align}
        \n{\nabla^m{\mathcal{V}}_j[v]}_{L^{p}}
        &
        \leq 
        C
        \lambda_{j-1}^4
        \n{\nabla^{m+2}\rho_j}_{L^1}
        \sum_{k \in \mathcal{K}}
        \n{\Gamma_{j,k}[v]\Psi_{j,k}}_{L^p}
        \\
        &
        \leq 
        C
        \lambda_{j-1}^4
        \lambda_j^{2m+4}
        \sum_{k \in \mathcal{K}}
        \n{\Psi_{j,k}}_{L^p}.
    \end{align}
    Since $\Psi_{j,k}$ is bounded as
    \begin{align}
        \n{\Psi_{j,k}}_{L^p}
        \leq 
        \sum_{\ell \in \Lambda_j}
        \frac{C\lambda^{\frac{n}{p}j}}{\lambda^{2\ell}\sqrt{h_j\ell}}
        \leq 
        \frac{C\lambda^{\frac{n}{p}j}}{\lambda^{4\mu}_{j-1}\sqrt{\mu\log_{\lambda}\lambda_{j-1}}}
        \leq 
        C
        \lambda^{n-4\mu}_{j-1}
    \end{align}
    we obtain \eqref{est:V[v]-1}.
    For the estimate of \eqref{est:V[v]-2}, \eqref{est:V[v]-1} with $p=\infty$ and $m=1$ yields
    \begin{align}
        \n{\mathcal{D}\mathcal{V}_j[v]}_{L^\infty}
        \leq 
        C\n{\nabla\mathcal{V}_j[v]}_{L^\infty}
        \leq 
        C\lambda_j^6
        \leq 
        r_0\lambda_j^{8}
    \end{align}
    provided that $\lambda$ is sufficiently large.
    Thus, we complete the proof.
\end{proof}
We now define the building blocks for the velocity field required in Theorem~\ref{thm:p-sol}.
\begin{df}\label{df:V_j}
    Let us define a sequence $\{{V}_j\}_{j \in \N_0} \subset C_c^{\infty}(\R^n)$ iteratively by 
    \begin{align}
        \begin{cases}
            {V}_0:=0, & (j=0), \\
            {V}_j:={\mathcal{V}}_j[{V}_{j-1}], \quad & (j = 1,2,3,\dots).
        \end{cases}
    \end{align}
\end{df}
\begin{rem}
    Some remarks are in order.
    \begin{enumerate}[label=\arabic*.]
        \item 
        Note that \eqref{est:V[v]-2} guarantees that $V_{j-1}$ belongs to the domain of the operator $\mathcal{V}_j$, and hence the sequence $\{V_j\}_{j\in\N_0}$ is well-defined.
        \item 
        By $\div \mathcal{D}=-\Delta \mathbb{P}$ in the definition of $\mathcal{V}_j$, we see that each $V_j$ is solenoidal.
        \item 
        It follows from 
        \begin{align}
        \supp V_{j-1} 
        &
        \subset \supp \psi_{j-1} + B_{\R^n}(0,\lambda_{j-1}^{-2}) \\
        &
        \subset B_{\R^n}(0,\lambda^{j-1}+1+\lambda_{j-1}^{-2}) 
        \\
        &
        \subset B_{\R^n}(0,\lambda^{j})
        \end{align}
        that $\psi_j(x)=1$ holds for all $x \in \supp V_{j-1}$.
        This fact is to be used for the cancellation of the worst nonlinearity. 
    \end{enumerate}
\end{rem}
To investigate the precise structure of $\{V_j\}_{j \in \N_0}$,
we decompose it into the principal and remainder part as
\begin{align}
    V_{j}=V^{(\rm p)}_{j} + V^{({\rm r})}_{j},
    \qquad
    V^{(\rm r)}_{j}
    =
    V^{({\rm r},1)}_{j}
    +
    V^{({\rm r},2)}_{j}
    +
    V^{({\rm r},3)}_{j},
\end{align}
where we have set 
\begin{align}
    V^{(\rm p)}_{j}(x)
    :={}&
    \sqrt{2}
    \lambda_{j-1}^4
    \sum_{k \in \mathcal{K}}
    \sum_{\ell \in \Lambda_j}
    \frac{1}{\lambda^{2\ell}\sqrt{h_j\ell}}
    \Gamma_{j,k}[V_{j-1}](x)
    \psi_{j}(x)
    \div\mathcal{D}
    \sp{k\cos\sp{\lambda^{\ell}a_k \cdot x}}
    \\
    ={}&
    \sqrt{2}
    \lambda_{j-1}^4
    \sum_{k \in \mathcal{K}}
    \sum_{\ell \in \Lambda_j}
    \frac{1}{\sqrt{h_j\ell}}
    \Gamma_{j,k}[V_{j-1}](x)
    \psi_{j}(x)k
    \cos\sp{\lambda^{\ell}a_k \cdot x}
\end{align}
and
\begin{align}
    &
    V^{({\rm r},1)}_{j}(x)
    :={}
    \rho_{j}*V^{(\rm p)}_{j}(x)-V^{(\rm p)}_{j}(x),
    \\
    &
    \begin{aligned}
    V^{({\rm r},2)}_{j}(x)
    :={}&
    \sqrt{2}
    \lambda_{j-1}^4
    \rho_{j}*
    \sum_{k \in \mathcal{K}}
    \sum_{\ell \in \Lambda_j}
    \frac{1}{\lambda^{\ell}\sqrt{h_j\ell}}
    \mathcal{D}
    \sp{\Gamma_{j,k}[V_{j-1}]
    \psi_j
    k}
    \nabla
    \cos\sp{\lambda^\ell a_k \cdot x}
    \\
    ={}&
    -
    \sqrt{2}
    \lambda_{j-1}^4
    \rho_{j}*
    \sum_{k \in \mathcal{K}}
    \sum_{\ell \in \Lambda_j}
    \frac{1}{\lambda^{\ell}\sqrt{h_j\ell}}
    \mathcal{D}
    \sp{\Gamma_{j,k}[V_{j-1}]
    \psi_j
    k}a_k
    \sin\sp{\lambda^\ell a_k \cdot x},
    \end{aligned}
    \\
    &
    V^{({\rm r},3)}_{j}(x)
    :={}
    \sqrt{2}
    \lambda_{j-1}^4
    \rho_{j}*
    \sum_{k \in \mathcal{K}}
    \sum_{\ell \in \Lambda_j}
    \frac{1}{\lambda^{2\ell}\sqrt{h_j\ell}}\div \mathcal{D}
    \sp{\Gamma_{j,k}[V_{j-1}]
    \psi_j
    k}
    \cos\sp{\lambda^\ell a_k \cdot x}.
\end{align}
We also introduce a function 
\begin{align}
    \Phi_{j,k,\ell}(x)
    :=
    \frac{1}{\sqrt{h_j\ell}}
    \Gamma_{j,k}[V_{j-1}](x)
    \psi_{j}(x)k
    \cos\sp{\lambda^{\ell}a_k \cdot x}
\end{align}
so that $V^{(\p)}_j$ has the following simplified representation: 
\begin{align}
    V^{(\rm p)}_{j}(x)
    =
    \sqrt{2}
    \lambda_{j-1}^4
    \sum_{k \in \mathcal{K}}
    \sum_{\ell \in \Lambda_j}
    \Phi_{j,k,\ell}(x).
\end{align}
\begin{lemm}\label{lemm:V_j-Besov}
    Let $1 \leq p,q \leq \infty$ and let $s \in \R$ satisfy
    \begin{align}
        \frac{n}{p}-n<s<2n.
    \end{align}
    Then, there exists a positive constant $C=C(n,p,q,s)$ such that 
    \begin{align}
        \label{est:V^p}
        C^{-1}
        \lambda^{\frac{n}{p}j}
        \lambda_{j-1}^{4-2\mu}
        \sigma_j(s,q)
        \leq 
        \n{V^{({\rm p})}_{j}}_{\dB_{p,q}^{s}}
        \leq 
        C
        \lambda^{\frac{n}{p}j}
        \lambda_{j-1}^{4-2\mu}
        \sigma_j(s,q),
    \end{align}
    and 
    \begin{align}\label{est:V^r}
        \n{V^{({\rm r})}_{j}}_{\dB_{p,q}^{s}}
        \leq 
        C
        \lambda^{nj}
        \lambda_{j-1}^{6-2\mu}\sigma_j(s-1,q)
    \end{align}
    for all $j \in \N$,
    where the sequence $\{\sigma_j(s,q)\}_{j \in \N}$ have been defined by
    \begin{align}
        \sigma_j(s,q)
        :=
        \sp{
        \displaystyle \sum_{\ell \in \Lambda_j}
        \frac{\lambda^{sq\ell}}{\ell^{\frac{q}{2}}}
        }^{\frac{1}{q}}.
    \end{align}
\end{lemm}
\begin{rem}
    The assumption $s<2n$ is not optimal, while this is enough for our purpose.
\end{rem}
\begin{proof}
To begin with, we prepare the $L^\infty$ estimates for the derivatives of $\Gamma_{j,k}[V_{j-1}]$.
Since $\Gamma_{j,k}[V_{j-1}]$ is constant for $j=1$, 
we focus on the case of $j \geq 2$.
Making use of \eqref{est:V[v]-1} for $V_{j-1}=\mathcal{V}_{j-1}[V_{j-2}]$, it holds
$\n{\nabla^s V_{j-1}}_{L^{\infty}} \leq C\lambda_{j-1}^{2s+4}$ ($s \in \N_0$).
Thus, from the Fa\`a di Bruno formula, it follows that
\begin{align}
    \n{\nabla^{m}\Gamma_{j,k}[V_{j-1}]}_{L^{\infty}}
    &\leq 
    C
    \sum_{\ell=1}^{m}
    \sum_{s_1+\dots+s_\ell=m}
    \lambda_{j-1}^{-8\ell}
    \n{\nabla^{s_1}\mathcal{D}V_{j-1}}_{L^{\infty}}
    \cdots
    \n{\nabla^{s_\ell}\mathcal{D}V_{j-1}}_{L^{\infty}}
    \\
    &
    \leq 
    C
    \sum_{\ell=1}^{m}
    \sum_{s_1+\dots+s_\ell=m}
    \lambda_{j-1}^{-8\ell}
    \cdot
    {\lambda_{j-1}^{2(s_1+1)+4}}
    \cdots
    {\lambda_{j-1}^{2(s_\ell+1)+4}}
    \\
    &
    \leq
    C\lambda_{j-1}^{2m}.
\end{align}

Let us prove \eqref{est:V^p}.
It follows from Proposition \ref{prop:F-Besov} with $M=2n$ that 
\begin{align}
    \n{V^{({\rm p})}_j}_{\dB_{p,q}^{s}}
    &
    \leq 
    C
    \lambda_{j-1}^{4-2\mu}
    \sigma_j(s,q)
    \sum_{k \in \mathcal{K}}
    \n{\Gamma_{j,k}[V_{j-1}]\psi_{j}}_{L^{p}}
    \\
    &\quad
    +
    C
    \lambda_{j-1}^{4-2\mu}
    \sum_{\ell\in \Lambda_j}
    \frac{\lambda^{(s-n)\ell}}{\sqrt{\ell}}
    \sum_{k \in \mathcal{K}}
    \n{\nabla^{2n}(\Gamma_{j,k}[V_{j-1}]\psi_{j})}_{L^{p} \cap L^1}
    \\
    &
    \leq 
    C\lambda^{\frac{n}{p}j}
    \lambda_{j-1}^{4-2\mu}
    \sigma_j(s,q)
    \sp{
    1+\lambda^{n(1-\frac{n}{p})j}\lambda_{j-1}^{(4-\mu)n}
    }
    \\
    &
    \leq 
    C
    \lambda^{\frac{n}{p}j}
    \lambda_{j-1}^{4-2\mu}
    \sigma_j(s,q),
\end{align}
where we have used 
\begin{align}
    &
    \sum_{\ell\in \Lambda_j}
    \frac{\lambda^{(s-n)\ell}}{\sqrt{\ell}}
    \leq 
    \sp{
    \sum_{\ell \in \Lambda_j}
    \lambda^{-nq'\ell}
    }^{\frac{1}{q'}}
    \sigma_j(s,q)
    \leq 
    \lambda_{j-1}^{-n\mu}
    \sigma_j(s,q)
\end{align}
and 
\begin{align}
    \n{\nabla^{2n}(\Gamma_{j,k}[V_{j-1}]\psi_{j})}_{L^{p} \cap L^1}
    \leq 
    C\lambda^{nj}\lambda_{j-1}^{4n}.
\end{align}
For the estimate of $V_j^{(\p)}$ from below, 
it holds by Proposition \ref{prop:F-Besov} and similar calculations as above that 
\begin{align}
    \n{V^{({\rm p})}_{j}}_{\dB_{p,q}^{s}}
    &
    \geq 
    C^{-1}
    \lambda_{j-1}^{4-2\mu}
    \sigma_j(s,q)
    \max_{k \in \mathcal{K}}
    \n{\Gamma_{j,k}[V_{j-1}]\psi_{j}}_{L^{p}}
    \\
    &\quad
    -
    C
    \lambda_{j-1}^{4-2\mu}
    \sum_{\ell\in \Lambda_j}
    \frac{\lambda^{(s-1)\ell}}{\sqrt{\ell}}
    \sum_{k \in \mathcal{K}}
    \n{\nabla(\Gamma_{j,k}[V_{j-1}]\psi_{j})}_{L^{p}}
    \\
    &\quad
    -
    C
    \lambda_{j-1}^{4-2\mu}
    \sum_{\ell\in \Lambda_j}
    \frac{\lambda^{(s-n)\ell}}{\sqrt{\ell}}
    \sum_{k \in \mathcal{K}}
    \n{\nabla^{2n}(\Gamma_{j,k}[V_{j-1}]\psi_{j})}_{L^{p} \cap L^1}
    \\
    &
    \geq 
    C^{-1}
    \lambda^{\frac{n}{p}j}
    \lambda_{j-1}^{4-2\mu}
    \sigma_j(s,q)
    \sp{
    1-C^2\lambda^{n(1-\frac{1}{p})j}\lambda_{j-1}^{2-\mu}-C^2\lambda^{n(1-\frac{1}{p})j}\lambda_{j-1}^{(4-\mu)n}
    }
    \\
    &
    \geq 
    C^{-1}
    \lambda^{\frac{n}{p}j}
    \lambda_{j-1}^{4-2\mu}
    \sigma_j(s,q).
\end{align}

Next, we show \eqref{est:V^r}.
It follows from $\int_{\R^n} \rho_{j}(y) dy = 1$ and the mean value formula that 
\begin{align}
    \Delta_{\ell}V^{({\rm r},1)}_{j}(x)
    &
    =
    \int_{\R^n}
    \rho_{j}(y)\sp{\Delta_{\ell}V^{({\rm p})}_{j}(x-y)-\Delta_{\ell}V^{({\rm p})}_j(x)}dy
    \\
    &
    =
    -
    \sum_{i=1}^n
    \int_{\R^n}
    y_i\rho_{j}(y)\int_0^1\partial_{x_i}\Delta_{\ell}V^{({\rm p})}_j(x-\theta y)d\theta dy,
\end{align}
which implies 
\begin{align}
    \n{\Delta_{\ell}V^{({\rm r},1)}_j}_{L^p}
    &
    \leq 
    \int_{\R^n}|y|\rho_{j}(y)dy
    \n{\nabla\Delta_{\ell}V^{({\rm p})}_j}_{L^p}
    = 
    C\lambda_j^{-2}
    \n{\nabla\Delta_{\ell}V^{({\rm p})}_j}_{L^p}.
\end{align}
Thus, we have 
\begin{align}
    \n{V^{({\rm r},1)}_j}_{\dB_{p,q}^s}
    &
    \leq 
    C
    \lambda_j^{-2}
    \n{V^{({\rm p})}_j}_{\dB_{p,q}^{s+1}}
    \\
    &\leq 
    C\lambda_j^{-2}
    \lambda^{\frac{n}{p}j}
    \lambda_{j-1}^{4-2\mu}\sigma_j(s+1,q)
    \\
    &\leq 
    C
    \lambda^{\frac{n}{p}j}
    \lambda_{j-1}^{4-2\mu}\sigma_j(s-1,q).\label{est:Vr1}
\end{align}
We follow the same argument as in the proof of \eqref{est:V^p} to obtain  
\begin{align}
    \n{V^{({\rm r},2)}_j}_{\dB_{p,q}^{s}}
    &
    \leq 
    C
    \lambda_{j-1}^{4-2\mu}
    \sigma_j(s-1,q)
    \sum_{k \in \mathcal{K}}
    \n{\nabla(\Gamma_{j,k}[V_{j-1}]\psi_{j})}_{L^{p}}
    \\
    &\quad
    +
    C
    \lambda_{j-1}^{4-2\mu}
    \sum_{\ell\in \Lambda_j}
    \frac{\lambda^{(s-1-n)\ell}}{\sqrt{\ell}}
    \sum_{k \in \mathcal{K}}
    \n{\nabla^{2n+1}(\Gamma_{j,k}[V_{j-1}]\psi_{j})}_{L^{p} \cap L^1}
    \\
    &
    \leq 
    C
    \lambda^{nj}
    \lambda_{j-1}^{6-2\mu}
    \sigma_j(s-1,q)
    \sp{
    1+\lambda_{j-1}^{(4-\mu)n}
    }
    \\
    &
    \leq 
    C\lambda^{nj}
    \lambda_{j-1}^{6-2\mu}
    \sigma_j(s-1,q)\label{est:Vr2}
\end{align}
and
\begin{align}
    \n{V^{({\rm r},3)}_j}_{\dB_{p,q}^{s}}
    &
    \leq 
    C
    \lambda_{j-1}^{4-2\mu}
    \sigma_j(s-2,q)
    \sum_{k \in \mathcal{K}}
    \n{\nabla^2(\Gamma_{j,k}[V_{j-1}]\psi_{j})}_{L^{p}}
    \\
    &\quad
    +
    C
    \lambda_{j-1}^{4-2\mu}
    \sum_{\ell\in \Lambda_j}
    \frac{\lambda^{(s-2-n)\ell}}{\sqrt{\ell}}
    \sum_{k \in \mathcal{K}}
    \n{\nabla^{2n+2}(\Gamma_{j,k}[V_{j-1}]\psi_{j})}_{L^{p} \cap L^1}
    \\
    &
    \leq 
    C
    \lambda^{nj}
    \lambda_{j-1}^{8-2\mu}
    \sigma_j(s-2,q)
    \sp{
    1+\lambda_{j-1}^{(4-\mu)n}
    }
    \\
    &
    \leq 
    C
    \lambda^{nj}
    \lambda_{j-1}^{8-3\mu}
    \sigma_j(s-1,q).\label{est:Vr3}
\end{align}
Gathering \eqref{est:Vr1}, \eqref{est:Vr2}, and \eqref{est:Vr3},
we obtain \eqref{est:V^r}
and complete the proof.
\end{proof}

\subsection{Proof of Theorem \ref{thm:p-sol}}
Now, we are in a position to prove Theorem \ref{thm:p-sol}. 
We separate the proof into two steps.
%

\noindent
{\it Step 1. Construction and estimates of the principal stationary flow.}
%
It follows from Lemmas \ref{lemm:V_j-Besov} that 
\begin{align}\label{est:ser-V^p}
    \n{V_j^{(\p)}}_{\dB_{p,q}^{\frac{n}{p}-1}}
    &
    \leq
    C
    \lambda^{\frac{n}{p}j}
    \lambda_{j-1}^{4-2\mu}
    \sigma_j\sp{\frac{n}{p}-1,q}
    \\
    &
    \leq 
    \begin{cases}
    C
    \lambda^{\frac{n}{p}j}
    \lambda_{j-1}^{4-2\mu}
    \sp{\mu\log_{\lambda}\lambda_{j-1}}^{\frac{1}{q}-\frac{1}{2}}
    & (p=n,\ q>2),
    \\[5pt]
    C
    \lambda^{\frac{n}{p}j}
    \lambda_{j-1}^{4-2\mu}
    \lambda_{j-1}^{(\frac{n}{p}-1)\mu}\sp{\mu\log_{\lambda}\lambda_{j-1}}^{-\frac{1}{2}} & (p>n)
    \end{cases}
    \\
    &
    \leq
    C
    \lambda^{\frac{n}{p}j}
    \lambda_{j-1}^{4-(3-\frac{n}{p})\mu}.
\end{align}
Hence, the series 
\begin{align}
    V^{(\p)}
    :=
    \sum_{j=1}^{\infty}V_j^{(\p)}
\end{align}
absolutely converges in the strong topology of $\dB_{p,q}^{n/p-1}(\R^n)$ and we see that 
\begin{align}
    \n{V^{(\p)}}_{\dB_{p,q}^{\frac{n}{p}-1}} 
    \leq 
    \sum_{j=1}^{\infty}
    \n{V_j^{(\p)}}_{\dB_{p,q}^{\frac{n}{p}-1}} 
    \leq 
    C\lambda^{\frac{n}{p}+4-(3-\frac{n}{p})\mu}.
\end{align}
For the estimate of $V^{(\p)}$ from below, it holds 
\begin{align}
    \n{V^{(\p)}}_{\dB_{p,q}^{\frac{n}{p}-1}}
    \geq{}& 
    \n{V_1^{(\p)}}_{\dB_{p,q}^{\frac{n}{p}-1}}
    -
    \sum_{j=2}^{\infty}
    \n{V_j^{(\p)}}_{\dB_{p,q}^{\frac{n}{p}-1}}
    \\
    \geq{}&
    C^{-1}
    \lambda^{\frac{n}{p}}
    \lambda^{4-2\mu}\sigma_1\sp{\frac{n}{p}-1,q}
    -
    C
    \sum_{j=2}^{\infty}
    \lambda^{\frac{n}{p}j}
    \lambda_{j-1}^{4-(3-\frac{n}{p})\mu}
    \\
    \geq{}& 
    C^{-1}
    \lambda^{\frac{n}{p}}
    \lambda^{4-(3-\frac{n}{p})\mu}
    -
    C
    \lambda^{\frac{2n}{p}}
    \lambda_{1}^{4-(3-\frac{n}{p})\mu}
    \\
    \geq{}& 
    C^{-1}
    \lambda^{\frac{n}{p}+4-(3-\frac{n}{p})\mu}
\end{align}
for sufficiently large $\lambda$.

Next, we consider the remainder part.
Let $\rho>n/2$.
Then, similarly as above, we see that the series 
\begin{align}
    V^{(\r)}
    :=
    \sum_{j=1}^{\infty}V_j^{(\r)}
\end{align}
absolutely converges in the strong topology of $\dB_{\rho,1}^{n/\rho-1}(\R^n)$ and 
\begin{align}
    \n{V^{(\r)}}_{\dB_{\rho,1}^{n/\rho-1}} 
    &\leq 
    \sum_{j=1}^{\infty}
    \n{V_j^{(\r)}}_{\dB_{\rho,1}^{\frac{n}{\rho}-1}}
    \\
    &
    \leq 
    C
    \sum_{j=1}^{\infty}
    \lambda^{nj}
    \lambda_{j-1}^{6-2\mu}
    \sigma_j\sp{\frac{n}{\rho}-2,1}
    \\
    &
    \leq 
    C
    \sum_{j=1}^{\infty}
    \lambda^{nj}
    \lambda_{j-1}^{6-(4-\frac{n}{\rho})\mu}
    \\
    &
    \leq 
    C\lambda^{n+6-(4-\frac{n}{\rho})\mu}
    \\
    &
    =
    C
    \lambda^{n(1-\frac{1}{p})+2-(1+\frac{n}{p}-\frac{n}{\rho})\mu}
    \lambda^{\frac{n}{p}+4-(3-\frac{n}{p})\mu}.
\end{align}
Now, we define $V:=V^{(\p)}+V^{(\r)}$.
Then, we see that 
\begin{align}
    \n{V}_{\dB_{p,q}^{\frac{n}{p}-1}}
    &
    \leq 
    \n{V^{(\p)}}_{\dB_{p,q}^{\frac{n}{p}-1}}
    +
    C
    \n{V^{(\r)}}_{\dB_{p,1}^{\frac{n}{p}-1}}
    \\
    &
    \leq 
    C
    \sp{
    1
    +
    \lambda^{n(1-\frac{1}{p})+2-\mu}
    }
    \lambda^{\frac{n}{p}+4-(3-\frac{n}{p})\mu}
    \\
    &
    \leq 
    C
    \lambda^{\frac{n}{p}+4-(3-\frac{n}{p})\mu}
    =
    C\lambda^{-\alpha(p)},
\end{align}
and
\begin{align}
    \n{V}_{\dB_{p,q}^{\frac{n}{p}-1}}
    &
    \geq 
    \n{V^{(\p)}}_{\dB_{p,q}^{\frac{n}{p}-1}}
    -
    C
    \n{V^{(\r)}}_{\dB_{p,1}^{\frac{n}{p}-1}}
    \\
    &
    \geq 
    C^{-1}
    \sp{
    1
    -
    C^2
    \lambda^{n(1-\frac{1}{p})+2-\mu}
    }
    \lambda^{\frac{n}{p}+4-(3-\frac{n}{p})\mu}
    \\
    &
    \geq 
    C^{-1}
    \lambda^{\frac{n}{p}+4-(3-\frac{n}{p})\mu}
    =
    C^{-1}\lambda^{-\alpha(p)},
\end{align}
provided that $\lambda$ is chosen sufficiently large.
Thus, we obtain \eqref{est:V-optimal}.
%

\noindent
{\it Step 2. Estimates of the perturbed force.}
%
In order to gain a better regularity and estimate for the force 
\begin{align}
    F=-\Delta V + \mathbb{P} \div (V \otimes V),
\end{align}
we focus on its precise structure.
To this end, we decompose $F$ as
\begin{align}
    &
    F
    =
    F_1 + F_2 + F_3;
    \\
    &
    \quad
    F_1
    :={}
    \sum_{j=1}^{\infty}
    \sp{-\Delta V_{j-1} + \mathbb{P} \div \sp{V^{({\rm p})}_j \otimes V^{({\rm p})}_j}},
    \\
    &
    \quad
    F_2
    :=
    \sum_{j_1 \neq j_2} \mathbb{P} \div \sp{ V^{({\rm p})}_{j_1} \otimes V^{({\rm p})}_{j_2}},
    \\
    &
    \quad
    F_3
    :=
    \mathbb{P} \div \sp{ V^{({\rm p})} \otimes V^{({\rm r})} + V^{({\rm r})} \otimes V^{({\rm p})} + V^{({\rm r})} \otimes V^{({\rm r})}}.
\end{align}
We first consider the estimate for $F_1$, which includes the most crucial cancellation in our analysis.
Using the formula $2\cos^2\theta = 1+ \cos 2 \theta$ and
\begin{align}
    \lambda_{j-1}^{8}
    \psi_{j}^2
    \sum_{k \in \mathcal{K}}
    \Gamma_k\sp{\Id-\lambda_{j-1}^{-8}\mathcal{D}V_{j-1}}^2
    k \otimes k
    &
    =
    \lambda_{j-1}^{8}
    \psi_{j}^2
    \sp{\Id-\lambda_{j-1}^{-8}\mathcal{D}V_{j-1}}
    \\
    &
    =
    \lambda_{j-1}^{8}\psi_{j}^2\Id
    -
    \mathcal{D}V_{j-1},
\end{align}
which is implied by Lemma \ref{lemm:Nash} and $\psi_{j}(x)=1$ for $x \in \supp V_{j-1}$,
we have
\begin{align}
    &
    V^{({\rm p})}_j(x) \otimes V^{({\rm p})}_j(x)
    ={}
    2\lambda_{j-1}^{8}
    \sum_{k_1, k_2 \in \mathcal{K}}
    \sum_{\ell_1,\ell_2 \in \Lambda_j}
    \Phi_{j,k_1,\ell_1}(x)
    \otimes
    \Phi_{j,k_2,\ell_2}(x)\\
    &\quad
    ={}
    \lambda_{j-1}^{8}
    \psi_{j}^2
    \sum_{k \in \mathcal{K}}
    \Gamma_k\sp{\Id-\lambda_{j-1}^{-8}\mathcal{D}V_{j-1}(x)}^2
    k \otimes k
    \\
    &\qquad
    +
    \frac{\lambda_{j-1}^{8}}{h_j}
    \sum_{k \in \mathcal{K}}
    \sum_{\ell \in \Lambda_j}
    \frac{1}{\ell}
    \Gamma_k\sp{\Id-\lambda_{j-1}^{-8}\mathcal{D}V_{j-1}(x)}^2
    \psi_{j}^2(x)
    \cos\sp{2\lambda^{\ell}a_k\cdot x}
    k \otimes k\\
    &\qquad
    +
    2\lambda_{j-1}^{8}
    \sum_{\substack{k_1,k_2 \in \mathcal{K} \\ k_1 \neq k_2}}
    \sum_{\ell_1,\ell_2 \in \Lambda_j}
    \Phi_{j,k_1,\ell}(x)
    \otimes 
    \Phi_{j,k_2,\ell}(x)\\
    &\qquad
    +
    2\lambda_{j-1}^{8}
    \sum_{k \in \mathcal{K}}
    \sum_{\substack{\ell_1,\ell_2 \in \Lambda_j \\ \ell_1 \neq \ell_2}}
    \Phi_{j,k,\ell_1}(x)
    \otimes 
    \Phi_{j,k,\ell_2}(x)\\
    &\quad
    ={}
    \lambda_{j-1}^{8}\psi_{j}(x)^2\Id
    -
    \mathcal{D}V_{j-1} (x)
    +F_{1,1;j}(x)
    +F_{1,2;j}(x)
    +F_{1,3;j}(x),
\end{align}
where we have set
\begin{align}
    &
    F_{1,1;j}(x)
    :=
    \frac{\lambda_{j-1}^{8}}{h_j}
    \sum_{k \in \mathcal{K}}
    \sum_{\ell \in \Lambda_j}
    \frac{1}{\ell}
    \Gamma_{j,k}\lp{V_{j-1}}(x)^2
    \psi_{j}^2
    \cos\sp{2\lambda^{\ell}a_k\cdot x}
    k \otimes k,
    \\
    &
    F_{1,2;j}(x)
    :=
    2\lambda_{j-1}^{8}
    \sum_{\substack{k_1,k_2 \in \mathcal{K} \\ k_1 \neq k_2}}
    \sum_{\ell_1,\ell_2 \in \Lambda_j}
    \Phi_{j,k_1,\ell_1}(x)
    \otimes
    \Phi_{j,k_2,\ell_2}(x),
    \\
    &
    F_{1,3;j}(x)
    :=
    2\lambda_{j-1}^{8}
    \sum_{k \in \mathcal{K}}
    \sum_{\substack{\ell_1,\ell_2 \in \Lambda_j \\ \ell_1 \neq \ell_2}}
    \Phi_{j,k,\ell_1}(x)
    \otimes 
    \Phi_{j,k,\ell_2}(x).
\end{align}
Then, from the cancellation property
\begin{align}
    -\Delta V_{j-1} + \mathbb{P}\div \sp{\lambda_{j-1}^8\psi_{j}^2\Id - \mathcal{D}V_{j-1}}
    =
    -\Delta V_{j-1} + \mathbb{P}\nabla (\lambda_{j-1}^8\psi_{j}^2) + \Delta V_{j-1}
    =
    0
\end{align}
it follows that
\begin{align}
    F_1
    &
    =
    \sum_{j=1}^{\infty}\sp{-\Delta V_{j-1} + \mathbb{P}\div \sp{\lambda_{j-1}^8\psi_{j}^2\Id - \mathcal{D}V_{j-1}
    +F_{1,1;j}
    +F_{1,2;j}
    +F_{1,3;j}}}
    \\
    &
    =
    \sum_{j=1}^{\infty}
    \mathbb{P}\div F_{1,1;j}
    +
    \sum_{j=1}^{\infty}
    \mathbb{P}\div F_{1,2;j}
    +
    \sum_{j=1}^{\infty}
    \mathbb{P}\div F_{1,3;j}
    \\
    &
    =: F_{1,1} + F_{1,2} + F_{1,3}.
\end{align}
To proceed with the estimate of $F_{1,1}$, 
we begin with the estimate of each element $F_{1,1;j}$.
It follows from Proposition \ref{prop:F-Besov} with $M=n$ that 
\begin{align}
    \n{F_{1,1;j}}_{\dB_{r,1}^{\frac{n}{r}-2}}
    &
    \leq 
    C\lambda_{j-1}^{8-4\mu}
    \sum_{\ell \in \lambda_j}\frac{\lambda^{(\frac{n}{r}-2)\ell}}{\ell}
    \sum_{k \in \mathcal{K}}
    \n{\Gamma_{j,k}[V_{j-1}]\psi_{j}}_{L^r}
    \\
    &\quad
    +
    C\lambda_{j-1}^{8-4\mu}
    \sum_{\ell \in \lambda_j}\frac{\lambda^{-2\ell}}{\ell}
    \sum_{k \in \mathcal{K}}
    \n{\nabla^n(\Gamma_{j,k}[V_{j-1}]\psi_{j})}_{L^r\cap L^1}
    \\
    &
    \leq 
    C
    \lambda_{j-1}^{8-4\mu}\cdot C\lambda_{j-1}^{(\frac{n}{r}-2)\mu}\cdot C\lambda^{\frac{n}{r}j}
    +
    C\lambda_{j-1}^{8-4\mu}\cdot C\lambda_{j-1}^{-2\mu}\cdot C\lambda^{nj}\lambda_{j-1}^{2n}
    \\
    &
    \leq 
    C\lambda_{j-1}^{4n+8-4\mu}.
\end{align}
For the estimate of $F_{1,2;j}$, since this part is rewritten as
\begin{align}
    F_{1,2;j}(x)
    =
    \frac{\lambda_{j-1}^{8}}{h_j}
    \sum_{\pm}
    \sum_{\ell_1,\ell_2 \in \Lambda_j}
    \sum_{\substack{k_1 \neq k_2}}
    &
    \frac{1}{\sqrt{\ell_1\ell_2}}
    \Gamma_{j,k_1}[V_{j-1}](x)
    \Gamma_{j,k_2}[V_{j-1}](x)
    \\
    &
    \psi_{j}(x)^2
    \cos\sp{(\lambda^{\ell_1}a_{k_1}\pm\lambda^{\ell_2}a_{k_2})\cdot x}
    k_1 \otimes k_2,
\end{align}
we see that 
\begin{align}
    \n{F_{1,2;j}}_{\dB_{r,1}^{\frac{n}{r}-2}}
    \leq
    {}&
    C
    \lambda_{j-1}^{8-4\mu}
    \sum_{\pm}
    \sum_{\substack{\ell_1 \in \Lambda_j}}
    \sum_{\ell_2=\ell_1}^{\infty}
    \sum_{\substack{k_1 \neq k_2}}
    \\
    &
    \left ( \frac{\lambda^{(\frac{n}{r}-2)\ell_2}}{\sqrt{\ell_1\ell_2}}
    |\lambda^{\ell_1-\ell_2}a_{k_1}\pm a_{k_2}|^{\frac{n}{r}-2}
    \n{\Gamma_{j,k_1}[V_{j-1}]
    \Gamma_{j,k_2}[V_{j-1}]
    \psi_{j}^2}_{L^r} \right. 
    \\
    & 
    \left. 
    \quad
    +
    \frac{\lambda^{-2\ell_2}}{\sqrt{\ell_1\ell_2}}
    |\lambda^{\ell_1-\ell_2}a_{k_1}\pm a_{k_2}|^{-2}
    \n{\nabla^n(\Gamma_{j,k_1}[V_{j-1}]
    \Gamma_{j,k_2}[V_{j-1}]
    \psi_{j}^2)}_{L^r\cap L^1} \right)
    \\
    \leq{}& 
    C
    \lambda_{j-1}^{8-4\mu}\cdot C\lambda_{j-1}^{(\frac{n}{r}-2)\mu}\cdot C\lambda^{\frac{n}{r}}
    +
    C\lambda_{j-1}^{8-4\mu}\cdot C\lambda_{j-1}^{-2\mu}\cdot C\lambda^n\lambda_{j-1}^{2n}
    \\
    \leq{}& 
    C\lambda_{j-1}^{4n+8-4\mu},
\end{align}
where we have used 
\begin{align}
    |\lambda^{\ell_1-\ell_2}a_{k_1}\pm a_{k_2}|^2
    ={}&
    \sp{\lambda^{\ell_1-\ell_2}\pm \cos \vartheta(a_{k_1},a_{k_2})}^2    
    +
    \sin^2\vartheta(a_{k_1},a_{k_2})
    \\
    \geq {}&
    \sin^2\vartheta(a_{k_1},a_{k_2})>0.
\end{align}
Here, one may recall \eqref{df:theta} for the definition of $\vartheta(\cdot,\cdot)$.
We may treat $F_{1,3;j}$ similarly.
Indeed, as there holds
\begin{align}
    F_{1,3;j}(x)
    =
    \frac{\lambda_{j-1}^{8}}{h_j}
    \sum_{\pm}
    \sum_{\ell_1 \neq \ell_2}
    \sum_{k \in \mathcal{K}}
    &
    \frac{1}{\sqrt{\ell_1\ell_2}}
    \Gamma_{j,k}[V_{j-1}](x)^2
    \\
    &
    \psi_{j}(x)^2
    \cos\sp{(\lambda^{\ell_1}\pm\lambda^{\ell_2})a_k\cdot x}
    k \otimes k,
\end{align}
we have
\begin{align}
    \n{F_{1,3;j}}_{\dB_{r,1}^{\frac{n}{r}-2}}
    \leq
    {}&
    C
    \lambda_{j-1}^{8-4\mu}
    \sum_{\pm}
    \sum_{k \in \mathcal{K}}
    \sum_{\substack{\ell_1 \in \Lambda_j}}
    \sum_{\ell_2=\ell_1+1}^{\infty}
    \\
    &
    \left ( \frac{\lambda^{(\frac{n}{r}-2)\ell_2}}{\sqrt{\ell_1\ell_2}}
    |(1\pm \lambda^{\ell_1-\ell_2})a_k|^{\frac{n}{r}-2}
    \n{\Gamma_{j,k}[V_{j-1}]^2
    \psi_{j}^2}_{L^r} \right. 
    \\
    & 
    \left. 
    \quad
    +
    \frac{\lambda^{-2\ell_2}}{\sqrt{\ell_1\ell_2}}
    |(1\pm \lambda^{\ell_1-\ell_2})a_k|^{-2}
    \n{\nabla^n(\Gamma_{j,k}[V_{j-1}]^2
    \psi_{j}^2)}_{L^r\cap L^1} \right)
    \\
    \leq{}& 
    C
    \lambda_{j-1}^{8-4\mu}\cdot C\lambda_{j-1}^{(\frac{n}{r}-2)\mu}\cdot C\lambda^{\frac{n}{r}}
    +
    C\lambda_{j-1}^{8-4\mu}\cdot C\lambda_{j-1}^{-2\mu}\cdot C\lambda^n\lambda_{j-1}^{2n}
    \\
    \leq{}& 
    C\lambda_{j-1}^{4n+8-4\mu}.
\end{align}
Hence, we obtain 
\begin{align}
    \n{F_{1}}_{\dB_{r,1}^{\frac{n}{r}-3}}
    &
    \leq
    C
    \sum_{j=1}^{\infty}
    \sp{
    \n{F_{1,1;j}}_{\dB_{r,1}^{\frac{n}{r}-2}}
    +
    \n{F_{1,2;j}}_{\dB_{r,1}^{\frac{n}{r}-2}}
    +
    \n{F_{1,3;j}}_{\dB_{r,1}^{\frac{n}{r}-2}}
    }
    \\
    &
    \leq
    C\sum_{j=1}^{\infty}
    \lambda_{j-1}^{4n+8-4\mu}
    \\
    &
    \leq{}
    C\lambda^{4n+8-4\mu}
    =
    C\lambda^{-\beta}.
    \label{est:F1}
\end{align}

Next, we consider the estimate on $F_2$.
Let us rewrite $F_2$ as 
\begin{align}
    F_2
    =
    \sum_{\substack{j_1,j_2 \in \N \\ j_1 \neq j_2}}
    \mathbb{P}\div F_{2;j_1,j_2},
\end{align}
where we have set
\begin{align}
    &
    F_{2;j_1,j_2}(x)
    :={}
    V^{({\rm p})}_{j_1} \otimes V^{({\rm p})}_{j_2}(x)
    \\
    &\quad
    =
    2\lambda_{j_1-1}^4\lambda_{j_2-1}^4
    \sum_{k_1, k_2 \in \mathcal{K}}
    \sum_{\ell_1,\ell_2 \in \Lambda_j}
    \Phi_{j_1,k_1,\ell_1}(x)
    \otimes
    \Phi_{j_2,k_2,\ell_2}(x)
    \\
    &\quad
    =
    \frac{\lambda_{j_1-1}^4\lambda_{j_2-1}^4}{\sqrt{h_{j_1}h_{j_2}}}
    \sum_{\pm}
    \sum_{\ell_1 \in \Lambda_{j_1}}
    \sum_{\ell_2 \in \Lambda_{j_2}}
    \sum_{k_1, k_2 \in \mathcal{K}}
    \frac{1}{\sqrt{\ell_1\ell_2}}
    \Gamma_{j_1,k_1}[V_{j_1-1}](x)
    \Gamma_{j_2,k_2}[V_{j_2-1}](x)
    \\
    & 
    \qquad \qquad \qquad \qquad \qquad \qquad 
    \psi_{j_1}(x)\psi_{j_2}(x)
    \cos\sp{(\lambda^{\ell_1}a_{k_1}\pm\lambda^{\ell_2}a_{k_2})\cdot x}k_1 \otimes k_2.
\end{align}
As $F_{2;j_1,j_2} = (F_{2;j_2,j_1})^{\top}$, it suffices to consider the case $j_1<j_2$.
Since it holds
\begin{align}
    |\lambda^{\ell_1-\ell_2}a_{k_1}\pm a_{k_2}|
    \geq 
    1 - \lambda^{\ell_1-\ell_2} 
    \geq 
    \frac{1}{2}
    >0,
\end{align}
we follow the similar strategy as for the estimate of $F_{1,2;j}$ to see that
\begin{align}
    \n{F_{2;j_1,j_2}}_{\dB_{r,1}^{\frac{n}{r}-2}}
    \leq{}&
    C\lambda_{j_1-1}^{4-2\mu}\lambda_{j_2-1}^{4-2\mu}
    \sum_{\ell_1 \in \Lambda_{j_1}}
    \sum_{\ell_2 \in \Lambda_{j_2}}
    \sum_{k_1, k_2 \in \mathcal{K}}
    \\
    &
    \left ( \frac{\lambda^{(\frac{n}{r}-2)\ell_2}}{\sqrt{\ell_1\ell_2}}
    |\lambda^{\ell_1-\ell_2}a_{k_1}\pm a_{k_2}|^{\frac{n}{r}-2}
    \n{\Gamma_{j_1,k_1}[V_{j_1-1}]
    \Gamma_{j_2,k_2}[V_{j_2-1}]
    \psi_{j_1}}_{L^r} \right. 
    \\
    & 
    \left. 
    \quad
    +
    \frac{\lambda^{-2\ell_2}}{\sqrt{\ell_1\ell_2}}
    |\lambda^{\ell_1-\ell_2}a_{k_1}\pm a_{k_2}|^{-2}
    \n{\nabla^n(\Gamma_{j_1,k_1}[V_{j_1-1}]
    \Gamma_{j_2,k_2}[V_{j_2-1}]
    \psi_{j_1})}_{L^r\cap L^1} \right)
    \\
    \leq{}&
    C
    \lambda^{\frac{n}{r}j_1}
    \lambda_{j_1-1}^{4-2\mu}\lambda_{j_2-1}^{4-2\mu}
    \sum_{\ell_1 \in \Lambda_{j_1}}
    \sum_{\ell_2 : \ell_2 \geq \ell_1}
    \lambda^{(\frac{n}{r}-2)\ell_2}
    \\
    &
    +
    C
    \lambda^{nj_1}
    \lambda_{j_1-1}^{4-2\mu}\lambda_{j_2-1}^{4-2\mu+2n}
    \sum_{\ell_1 \in \Lambda_{j_1}}
    \sum_{\ell_2: \ell_2 \geq \ell_1}
    \lambda^{-2\ell_2}
    \\
    \leq{}&
    C
    \lambda^{\frac{n}{r}j_1}
    \lambda_{j_1-1}^{4-(4-\frac{n}{r})\mu}\lambda_{j_2-1}^{4-2\mu}
    +
    C
    \lambda^{nj_1}
    \lambda_{j_1-1}^{4-4\mu}\lambda_{j_2-1}^{4-2\mu+2n}
    \\
    \leq{}&
    C
    \lambda^{nj_1}
    \lambda_{j_1-1}^{4-(4-\frac{n}{r})\mu}\lambda_{j_2-1}^{4-2\mu+2n}.
\end{align}
Hence, we obtain 
\begin{align}
    \n{F_{2}}_{\dB_{r,1}^{\frac{n}{r}-3}}
    &
    \leq
    C
    \sum_{j_1=1}^{\infty}
    \sum_{j_2=j_1+1}^{\infty}
    \n{F_{2;j_1,j_2}}_{\dB_{r,1}^{\frac{n}{r}-2}}
    \\
    &
    \leq
    C
    \sum_{j_1=1}^{\infty}
    \lambda^{nj_1}
    \lambda_{j_1-1}^{4-(4-\frac{n}{r})\mu}
    \sum_{j_2=j_1+1}^{\infty}
    \lambda_{j_2-1}^{4-2\mu+2n}
    \\
    &
    \leq
    C
    \sum_{j_1=1}^{\infty}
    \lambda^{nj_1}
    \lambda_{j_1-1}^{8+2n-(6-\frac{n}{r})\mu}
    \\
    &
    \leq
    C
    \lambda^{8+3n-(6-\frac{n}{r})\mu}
    \leq
    C
    \lambda^{4n+8-4\mu}
    =
    C\lambda^{-\beta}.
    \label{est:F2}
\end{align}
Finally, we focus on the estimate of $F_3$.
Let $n<p_*<2n$ satisfy \eqref{p-r} with $p=p_*$.
It then follows from Lemma \ref{lemm:nonlin-sta} that
\begin{align}
    \n{F_3}_{\dB_{r,1}^{\frac{n}{r}-3}}
    &
    \leq 
    C
    \sp{
    \n{V^{({\rm p})}}_{\dB_{p_*,1}^{\frac{n}{p_*}-1}}
    +
    \n{V^{({\rm r})}}_{\dB_{r,1}^{\frac{n}{r}-1}}
    }
    \n{V^{({\rm r})}}_{\dB_{r,1}^{\frac{n}{r}-1}}
    \\
    &\leq 
    C
    \sp{
    \lambda^{\frac{n}{p_*}+4-(3-\frac{n}{p_*})\mu}
    +
    \lambda^{n+6-(4-\frac{n}{r})\mu}
    }
    \lambda^{n+6-(4-\frac{n}{r})\mu}
    \\
    &\leq 
    C
    \sp{
    \lambda^{5-2\mu}
    +
    \lambda^{n+6-2\mu}
    }
    \lambda^{n+6-2\mu}
    \\
    &
    \leq 
    C
    \lambda^{2n+12-4\mu}
    <
    C
    \lambda^{4n+8-4\mu}
    =
    C\lambda^{-\beta}.
    \label{est:F3}
\end{align}
Combining \eqref{est:F1}, \eqref{est:F2}, and \eqref{est:F3}, we obtain \eqref{est:F}.
Thus, we complete the proof.

\section{Proof of Theorem \ref{thm:main}}\label{sec:pf}
Now, we are ready to present the proof of Theorem \ref{thm:main}.
\begin{proof}
We separate the proof into two steps; the first step achieves the construction of the non-trivial solutions to the unforced stationary Navier--Stokes equations, and the second step establishes the desired non-stationary flow by the standard stability argument around the steady state provided in the previous step.
Let $V_\lambda$ and $F_{\lambda}$ be vector fields constructed in Theorem \ref{thm:p-sol}.
%

\noindent
{\it Step 1. Steady Navier--Stokes flow around $V_{\lambda}$.}
%
Fix a $n<p_*<2n$ and $n/2<r<n$ so that \eqref{p-r} holds with $p$ replaced by $p_*$. 
Let $\lambda \geq \lambda_{*0}$, where $\lambda_{*0}$ is a constant appearing in Theorem \ref{thm:p-sol}.
Let us construct the solution to \eqref{eq:sNS} of the form $U_{\lambda}=V_{\lambda}+W_{\lambda}$.
To this end, we consider the stationary perturbed system
\begin{align}\label{eq:W}
    \begin{cases}
        -\Delta W_{\lambda} + \mathbb{P}\div \sp{V_{\lambda} \otimes W_{\lambda} + W_{\lambda} \otimes V_{\lambda} + W_{\lambda} \otimes W_{\lambda}} = - F_{\lambda}, \quad & x \in \R^n, \\
        \div W_{\lambda} = 0, \quad & x \in \R^n,
    \end{cases}
\end{align}
and define an operator $\mathcal{X}$ on $\dB_{r,1}^{n/r-1}(\R^n)$ by
\begin{align}
    \mathcal{X}[W]
    := - (-\Delta)^{-1}F_{\lambda} - (-\Delta)^{-1}\mathbb{P}\div \sp{V_{\lambda} \otimes W + W \otimes V_{\lambda}+ W \otimes W}.
\end{align}
Then, it follows from Lemma \ref{lemm:nonlin-sta} that 
\begin{align}
    &
    \begin{aligned}
    \n{\mathcal{X}[W]}_{\dB_{r,1}^{\frac{n}{r}-1}}
    &
    \leq 
    C
    \n{F_{\lambda}}_{\dB_{r,1}^{\frac{n}{r}-3}}
    +
    C
    \sp{
    \n{V_{\lambda}}_{\dB_{p_*,1}^{\frac{n}{p_*}-1}}
    +
    \n{W}_{\dB_{r,1}^{\frac{n}{r}-1}}
    }
    \n{W}_{\dB_{r,1}^{\frac{n}{r}-1}},
    \\
    &
    \leq 
    C_1
    \lambda^{-\beta}
    +
    C_1
    \sp{
    \lambda^{-\alpha(p_*)}
    +
    \n{W}_{\dB_{r,1}^{\frac{n}{r}-1}}
    }
    \n{W}_{\dB_{r,1}^{\frac{n}{r}-1}},
    \end{aligned}
    \\
    &
    \n{\mathcal{X}[W_1]-\mathcal{X}[W_2]}_{\dB_{r,1}^{\frac{n}{r}-1}}
    \\
    &\quad
    \leq 
    C
    \sp{
    \n{V_{\lambda}}_{\dB_{p_*,1}^{\frac{n}{p_*}-1}}
    +
    \n{W_1}_{\dB_{r,1}^{\frac{n}{r}-1}}
    +
    \n{W_2}_{\dB_{r,1}^{\frac{n}{r}-1}}
    }
    \n{W_1-W_2}_{\dB_{r,1}^{\frac{n}{r}-1}}
    \\
    &\quad
    \leq 
    C_1
    \sp{
    \lambda^{-\alpha(p_*)}
    +
    \n{W_1}_{\dB_{r,1}^{\frac{n}{r}-1}}
    +
    \n{W_2}_{\dB_{r,1}^{\frac{n}{r}-1}}
    }
    \n{W_1-W_2}_{\dB_{r,1}^{\frac{n}{r}-1}}
\end{align}
for all $W,W_1,W_2 \in \dB_{r,1}^{n/r-1}(\R^n)$ 
with some positive constant $C_1=C_1(n,p_*,r)$.
Set
\begin{align}
    \mathscr{X}:=
    \Mp{W \in \dB_{r,1}^{n/r-1}(\R^n)\ ;\ \n{W}_{\dB_{r,1}^{n/r-1}} \leq 2C_1\lambda^{-\beta}}.
\end{align}
Then,
choosing a constant $\lambda_{*1}$ so large that $\lambda_{*1} \geq \lambda_{*0}$ and $2C_1(\lambda_{*1}^{-\alpha(p_*)}+4C_1\lambda_{*1}^{-\beta}) \leq 1$, we have 
\begin{align}
    \n{\mathcal{X}[W]}_{\dB_{r,1}^{\frac{n}{r}-1}}
    \leq 
    \frac{3}{2}C_1\lambda^{-\beta},
    \qquad
    \n{\mathcal{X}[W_1]-\mathcal{X}[W_2]}_{\dB_{r,1}^{\frac{n}{r}-1}}
    \leq 
    \frac{1}{2}
    \n{W_1-W_2}_{\dB_{r,1}^{\frac{n}{r}-1}}
\end{align}
for all $W,W_1,W_2 \in \mathscr{X}$.
Hence, the contraction mapping principle yields the existence of a unique solution $W_{\lambda} \in \mathscr{X}$ to \eqref{eq:W}. 

Let $U_{\lambda}:=V_{\lambda}+W_{\lambda}$. Then, we see that $U_{\lambda}$ solves \eqref{eq:sNS}.
Moreover, it follows from Theorem \ref{thm:p-sol} and $W_{\lambda} \in \mathscr{X}$ that 
\begin{align}
    &
    \n{U_{\lambda}}_{\dB_{p,q}^{\frac{n}{p}-1}} 
    \leq 
    \n{V_{\lambda}}_{\dB_{p,q}^{\frac{n}{p}-1}}
    +
    C
    \n{W_{\lambda}}_{\dB_{r,1}^{\frac{n}{r}-1}}
    \leq 
    C\sp{\lambda^{-\alpha(p)}+\lambda^{-\beta}},
    \\
    &
    \n{U_{\lambda}}_{\dB_{p,q}^{\frac{n}{p}-1}} 
    \geq 
    \n{V_{\lambda}}_{\dB_{p,q}^{\frac{n}{p}-1}}
    -
    C
    \n{W_{\lambda}}_{\dB_{r,1}^{\frac{n}{r}-1}}
    \geq 
    C^{-1}\lambda^{-\alpha(p)}-C\lambda^{-\beta}.
\end{align}
Since $\alpha(p)<\beta$, there exist $\lambda_{*2} \geq \lambda_{*1}$ and $C_2>0$ such that 
\begin{align}\label{U_lambda}
    C_2^{-1}\lambda^{-\alpha(p)} \leq \n{U_\lambda}_{\dB_{p,q}^{\frac{n}{p}-1}} \leq C_2\lambda^{-\alpha(p)}
\end{align}
for all $\lambda \geq \lambda_{*2}$.
%

\noindent
{\it Step 2. Non-stationary Navier--Stokes flow around $U_{\lambda}$.}
%
Taking the smallness parameter $\eta=0$ for the case of $p \geq 2n$ or $q=\infty$,
and recalling the continuous embedding $\dB_{p_1,q_1}^{n/p_1-1}(\R^n) \hookrightarrow \dB_{p_2,q_2}^{n/p_2-1}(\R^n)$ ($1 \leq p_1 \leq p_2 \leq \infty$, $1 \leq q_1 \leq q_2 \leq \infty$),
we see that it suffices to consider only the case of $p<2n$ and $q<\infty$.
Then, we may choose a $\theta$ satisfying \eqref{p-r-theta}.
{We arbitrary fix a solenoidal initial datum $u_0$ satisfying
\begin{align}
    u_0 \in \dB_{p,q}^{\frac{n}{p}-1}(\R^n),
    \qquad
    \n{u_0}_{\dB_{p,q}^{\frac{n}{p}-1}}
    \leq \eta,
\end{align}
where $\eta$ is a positive constant to be determined later.
We shall prove that there exists a $\lambda_{*3} \geq \lambda_{*2}$ such that this fixed $u_0$ generates a family $\{ u_\lambda \}_{\lambda \geq \lambda_{*3}}$ of solutions to \eqref{eq:NS-intro} satisfying 
\begin{align}
    \lim_{t \to \infty} \n{u_{\lambda}(t) - U_{\lambda}}_{\dB_{p,q}^{\frac{n}{p}-1}} = 0,
\end{align}
where $U_{\lambda}$ is the solution to \eqref{eq:sNS} satisfying \eqref{U_lambda}, which was constructed in the previous step.
Once this is proved, we see that $u_0$ generates infinitely many solutions since $U_{\lambda}$ are distinct for each sufficiently large $\lambda$ due to \eqref{U_lambda}.}
To this end, we consider the perturbed system of \eqref{eq:NS-intro} around $U_{\lambda}$.
\begin{align}\label{eq:w}
    \begin{cases}
        \partial_t w_{\lambda} - \Delta w_{\lambda} + \mathbb{P}\div (U_{\lambda} \otimes w_{\lambda} + w_{\lambda} \otimes U_{\lambda} + w_{\lambda} \otimes w_{\lambda}) = 0, \\
        \div w_{\lambda} = 0,\\
        w_{\lambda}(0,x)=w_{0,\lambda}:=u_0-U_{\lambda}.
    \end{cases}
\end{align}
To construct the perturbed solution, we define a mapping 
\begin{align}
    \mathcal{Y}[w](t)
    :=
    e^{t\Delta}w_{0,\lambda}
    -\int_0^t e^{(t-\tau)\Delta}\mathbb{P}\div\sp{U_{\lambda} \otimes w(\tau) + w(\tau) \otimes U_{\lambda} + w(\tau) \otimes w(\tau)}d\tau.
\end{align}
Then, we see by Lemma \ref{lemm:nonlin-Duha} that 
\begin{align}
    &
    \n{\mathcal{Y}[w]}_{\widetilde{L^{\infty}}(0,\infty;\dB_{p,q}^{\frac{n}{p}-1})\cap\widetilde{L^{\theta}}(0,\infty;\dB_{p,q}^{\frac{n}{p}-1+\frac{2}{\theta}})}
    \\
    &\quad
    \leq{}
    C
    \n{w_{0,\lambda}}_{\dB_{p,q}^{\frac{n}{p}-1}}
    +
    C
    \sp{
    \n{U_{\lambda}}_{\dB_{p,q}^{\frac{n}{p}-1}}
    +
    \n{w}_{\widetilde{L^{\infty}}(0,\infty;\dB_{p,q}^{\frac{n}{p}-1})}
    }
    \n{w}_{\widetilde{L^{\theta}}(0,\infty;\dB_{p,q}^{\frac{n}{p}-1+\frac{2}{\theta}})}
    \\
    &\quad
    \leq{}
    C_3
    \n{w_{0,\lambda}}_{\dB_{p,q}^{\frac{n}{p}-1}}
    +
    C_3
    \sp{
    \lambda^{-\alpha(p)}
    +
    \n{w}_{\widetilde{L^{\infty}}(0,\infty;\dB_{p,q}^{\frac{n}{p}-1})}
    }
    \n{w}_{\widetilde{L^{\theta}}(0,\infty;\dB_{p,q}^{\frac{n}{p}-1+\frac{2}{\theta}})}
\end{align}
and
\begin{align}
    &
    \n{\mathcal{Y}[w_1]-\mathcal{Y}[w_2]}_{\widetilde{L^{\theta}}(0,\infty;\dB_{p,q}^{\frac{n}{p}-1+\frac{2}{\theta}})}
    \\
    &\quad
    \leq{}
    C
    \sp{
    \n{U_{\lambda}}_{\dB_{p,q}^{\frac{n}{p}-1}}
    +
    \sum_{m=1}^2\n{w_m}_{\widetilde{L^{\infty}}(0,\infty;\dB_{p,q}^{\frac{n}{p}-1})}
    }
    \n{w_1-w_2}_{\widetilde{L^{\theta}}(0,\infty;\dB_{p,q}^{\frac{n}{p}-1+\frac{2}{\theta}})}
    \\
    &\quad
    \leq{}
    C_3
    \sp{
    \lambda^{-\alpha(p)}
    +
    \sum_{m=1}^2\n{w_m}_{\widetilde{L^{\infty}}(0,\infty;\dB_{p,q}^{\frac{n}{p}-1})}
    }
    \n{w_1-w_2}_{\widetilde{L^{\theta}}(0,\infty;\dB_{p,q}^{\frac{n}{p}-1+\frac{2}{\theta}})}
\end{align}
for all $w,w_1,w_2 \in \widetilde{C}([0,\infty);\dB_{p,q}^{n/p-1}(\R^n))\cap \widetilde{L^{\theta}}(0,\infty;\dB_{p,q}^{n/p-1+2/\theta}(\R^n))$ with some positive constant $C_2=C_2(n,p,q,\theta)$.
We define the solution space $\mathscr{Y}$ by
\begin{align}
    \mathscr{Y}
    :=
    \Mp{
    \begin{aligned}
    w \in {}&
    \widetilde{C}([0,\infty);\dB_{p,q}^{\frac{n}{p}-1}(\R^n))
    \cap 
    \widetilde{L^{\theta}}(0,\infty;\dB_{p,q}^{\frac{n}{p}-1+\frac{2}{\theta}}(\R^n))
    \ ;\\
    &
    \n{w}_{\widetilde{L^{\infty}}(0,\infty;\dB_{p,q}^{\frac{n}{p}-1})\cap\widetilde{L^{\theta}}(0,\infty;\dB_{p,q}^{\frac{n}{p}-1+\frac{2}{\theta}})}
    \leq 
    2C_3\n{w_{0,\lambda}}_{\dB_{p,q}^{\frac{n}{p}-1}}
    \end{aligned}
    }.
\end{align}
Then for $w,w_1,w_2 \in \mathscr{Y}$, we have 
\begin{align}
    &
    \n{\mathcal{Y}[w]}_{\widetilde{L^{\infty}}(0,\infty;\dB_{p,q}^{\frac{n}{p}-1})\cap\widetilde{L^{\theta}}(0,\infty;\dB_{p,q}^{\frac{n}{p}-1+\frac{2}{\theta}})}
    \\
    &\quad
    \leq 
    C_3
    \n{w_{0,\lambda}}_{\dB_{p,q}^{\frac{n}{p}-1}}
    +
    C_3\sp{\lambda^{-\alpha(p)} + 4C_2\n{w_{0,\lambda}}_{\dB_{p,q}^{\frac{n}{p}-1}}}
    \n{w}_{\widetilde{L^{\theta}}(0,\infty;\dB_{p,q}^{\frac{n}{p}-1+\frac{2}{\theta}})},
    \\
    &
    \n{\mathcal{Y}[w_1]-\mathcal{Y}[w_2]}_{\widetilde{L^{\theta}}(0,\infty;\dB_{p,q}^{\frac{n}{p}-1+\frac{2}{\theta}})}
    \\
    &\quad
    \leq{}
    C_3\sp{\lambda^{-\alpha(p)} + 4C_2\n{w_{0,\lambda}}_{\dB_{p,q}^{\frac{n}{p}-1}}}
    \n{w_1-w_2}_{\widetilde{L^{\theta}}(0,\infty;\dB_{p,q}^{\frac{n}{p}-1+\frac{2}{\theta}})}
\end{align}
{
Choosing $\lambda_{*3}$ so large and $\eta$ so small that
\begin{align}
    {\lambda^{-\alpha(p)} + 4C_2\n{w_{0,\lambda}}_{\dB_{p,q}^{\frac{n}{p}-1}}}
    &
    =
    {\lambda^{-\alpha(p)} + 4C_2\n{u_0 - U_{\lambda}}_{\dB_{p,q}^{\frac{n}{p}-1}}}
    \\
    &
    \leq 
    {\lambda^{-\alpha(p)} + 4C_3\n{u_0}_{\dB_{p,q}^{\frac{n}{p}-1}} +4C_3\n{U_{\lambda}}_{\dB_{p,q}^{\frac{n}{p}-1}}}
    \\
    &
    \leq 
    {\lambda^{-\alpha(p)} + 4C_2\eta + 4C_2C_3\lambda^{-\alpha(p)}}
    \\
    &
    \leq \frac{1}{2C_3}
\end{align}
for all $\lambda \geq \lambda_{*3}$},
we have
\begin{align}
    &
    \n{\mathcal{Y}[w]}_{\widetilde{L^{\infty}}(0,\infty;\dB_{p,q}^{\frac{n}{p}-1})\cap\widetilde{L^{\theta}}(0,\infty;\dB_{p,q}^{\frac{n}{p}-1+\frac{2}{\theta}})}
    \leq 
    \frac{3}{2}C_2\n{w_0}_{\dB_{p,q}^{\frac{n}{p}-1}},
    \\
    &
    \n{\mathcal{Y}[w_1]-\mathcal{Y}[w_2]}_{\widetilde{L^{\theta}}(0,\infty;\dB_{p,q}^{\frac{n}{p}-1+\frac{2}{\theta}})}
    \leq 
    \frac{1}{2}
    \n{w_1-w_2}_{\widetilde{L^{\theta}}(0,\infty;\dB_{p,q}^{\frac{n}{p}-1+\frac{2}{\theta}})}
\end{align}
for all $w,w_1,w_2 \in \mathscr{Y}$,
which implies that
$\mathcal{Y}$ is a contraction mapping on $\mathscr{Y}$.
Hence, the Banach fixed point theorem implies that there exists a unique mild solution $w_{\lambda} \in \mathscr{Y}$ to the perturbed system \eqref{eq:w}.
We then obtain {the non-stationary solution $u_{\lambda}:=U_{\lambda}+w_{\lambda}$ to \eqref{eq:NS-intro} with the initial data $u_{\lambda}|_{t=0}=U_{\lambda} + w_{\lambda}|_{t=0}=u_0$.}

Finally, we shall prove the asymptotic behavior $u_{\lambda}(t)-U_{\lambda}=w_{\lambda}(t) \to 0$ strongly in $\dB_{p,q}^{n/p-1}(\R^n)$ as $t \to \infty$.
Let us fix a $T>0$ and rewrite the integral equation for $w_\lambda$ as 
\begin{align}
    w_{\lambda}(t)
    ={}&
    e^{(t-T)\Delta}w_{\lambda}(T)
    \\
    &
    -
    \int_T^t e^{(t-\tau)\Delta}\mathbb{P}\div\sp{U_{\lambda} \otimes w_{\lambda}(\tau) + w_{\lambda}(\tau) \otimes U_{\lambda} + w_{\lambda}(\tau) \otimes w_{\lambda}(\tau)}d\tau
\end{align}
for $t>T$.
Then, by Lemma \ref{lemm:nonlin-Duha}, we have for $T'>T$ that
\begin{align}
    \n{w_{\lambda}}_{L^{\infty}(T',\infty;\dB_{p,q}^{\frac{n}{p}-1})}
    \leq{}&
    \n{w_{\lambda}}_{\widetilde{L^{\infty}}(T',\infty;\dB_{p,q}^{\frac{n}{p}-1})}
    \\
    \leq{}&
    \n{e^{(T'-T)\Delta}w_{\lambda}(T)}_{\dB_{p,q}^{\frac{n}{p}-1}}
    \\
    &+
    C
    \sp{
    \n{U_{\lambda}}_{\dB_{p,q}^{\frac{n}{p}-1}}
    +
    \n{w_{\lambda}}_{\widetilde{L^{\infty}}(0,\infty;\dB_{p,q}^{\frac{n}{p}-1})}
    }
    \n{w_{\lambda}}_{\widetilde{L^{\theta}}(T,\infty;\dB_{p,q}^{\frac{n}{p}-1+\frac{2}{\theta}})}
    \\
    \leq{}&
    C
    \n{e^{(T'-T)\Delta}w_{\lambda}(T)}_{\dB_{p,q}^{\frac{n}{p}-1}}
    +
    C
    \n{w_{\lambda}}_{\widetilde{L^{\theta}}(T,\infty;\dB_{p,q}^{\frac{n}{p}-1+\frac{2}{\theta}})},
\end{align}
which implies
\begin{align}
    \limsup_{T' \to \infty}
    \n{w_{\lambda}}_{L^{\infty}(T',\infty;\dB_{p,q}^{\frac{n}{p}-1})}
    \leq 
    C
    \n{w_{\lambda}}_{\widetilde{L^{\theta}}(T,\infty;\dB_{p,q}^{\frac{n}{p}-1+\frac{2}{\theta}})}.
\end{align}
Hence, letting $T \to \infty$, we complete the proof.
\end{proof}

\noindent
{\bf Acknowledgments.} \\
The author was supported by JSPS KAKENHI, Grant Number JP25K17279. 

\noindent
{\bf Conflict of interest.} \\
The author has declared no conflicts of interest.

\noindent
{\bf Data availability.} \\
Data sharing not applicable to this article as no dataset was generated or analyzed during the current study.



\begin{thebibliography}{100}

\bib{Abidi-Paicu-2007}{article}{
   author={Abidi, Hammadi},
   author={Paicu, Marius},
   title={Existence globale pour un fluide inhomog\`ene},
   language={French, with English and French summaries},
   journal={Ann. Inst. Fourier (Grenoble)},
   volume={57},
   date={2007},
   pages={883--917},
}

\bib{Alb-Bru-Col-AnnMath}{article}{
   author={Albritton, Dallas},
   author={Bru\'e, Elia},
   author={Colombo, Maria},
   title={Non-uniqueness of Leray solutions of the forced Navier-Stokes
   equations},
   journal={Ann. of Math. (2)},
   volume={196},
   date={2022},
   pages={415--455},
}

\bib{Aok-Mae}{article}{
  author={Aoki, Motofumi},
  author={Maekawa, Yasunori},
  title={Self-similar vorticity around the boundary and non-uniqueness of solutions to the two-dimensional Navier-Stokes equations in the half space},
  note={Preprint},
  pages={arXiv:2507.02338},
  date={2025},
}


\bib{Bah-Che-Dan-11}{book}{
   author={Bahouri, Hajer},
   author={Chemin, Jean-Yves},
   author={Danchin, Rapha\"{e}l},
   title={Fourier analysis and nonlinear partial differential equations},
   series={Grundlehren der mathematischen Wissenschaften [Fundamental
   Principles of Mathematical Sciences]},
   volume={343},
   publisher={Springer, Heidelberg},
   date={2011},
   pages={xvi+523},
}

\bib{Benameu-2015-JMAA}{article}{
   author={Benameur, Jamel},
   title={Long time decay to the Lei-Lin solution of 3D Navier-Stokes
   equations},
   journal={J. Math. Anal. Appl.},
   volume={422},
   date={2015},
   pages={424--434},
}

\bib{Bou-Pav-08}{article}{
  author={Bourgain, Jean},
   author={Pavlovi\'c, Nata\v sa},
   title={Ill-posedness of the Navier-Stokes equations in a critical space
   in 3D},
   journal={J. Funct. Anal.},
   volume={255},
   date={2008},
   pages={2233--2247},
}

\bib{Buckmaster-Lellis-Isett-Szekelyhidi-2015-AnnMath}{article}{
   author={Buckmaster, Tristan},
   author={De Lellis, Camillo},
   author={Isett, Philip},
   author={Sz\'ekelyhidi, L\'aszl\'o, Jr.},
   title={Anomalous dissipation for $1/5$-H\"older Euler flows},
   journal={Ann. of Math. (2)},
   volume={182},
   date={2015},
   pages={127--172},
}

\bib{Buc-Vic-19}{article}{
  author={Buckmaster, Tristan},
   author={Vicol, Vlad},
   title={Nonuniqueness of weak solutions to the Navier-Stokes equation},
   journal={Ann. of Math. (2)},
   volume={189},
   date={2019},
   pages={101--144},
}

\bib{Can-Pla-96}{article}{
   author={Cannone, M.},
   author={Planchon, F.},
   title={Self-similar solutions for Navier-Stokes equations in ${\bf R}^3$},
   journal={Comm. Partial Differential Equations},
   volume={21},
   date={1996},
   pages={179--193},
}

\bib{Che-92}{article}{
   author={Chemin, J.-Y.},
   title={Remarques sur l'existence globale pour le syst\`eme de
   Navier-Stokes incompressible},
   language={French, with English summary},
   journal={SIAM J. Math. Anal.},
   volume={23},
   date={1992},
   pages={20--28},
}

\bib{Che-Ler-95}{article}{
   author={Chemin, J.-Y.},
   author={Lerner, N.},
   title={Flot de champs de vecteurs non lipschitziens et \'equations de
   Navier-Stokes},
   language={French},
   journal={J. Differential Equations},
   volume={121},
   date={1995},
   pages={314--328},
}

\bib{Che-93}{article}{
   author={Chen, Zhi Min},
   title={$L^n$ solutions of the stationary and nonstationary Navier-Stokes
   equations in ${\bf R}^n$},
   journal={Pacific J. Math.},
   volume={158},
   date={1993},
   pages={293--303},
}

\bib{Che-Hou-26}{article}{
  author={Cheskidov, Alexey},
  author={Hou, Hedong},
  title={On non-uniqueness of mild solutions and stationary singular solutions to the Navier--Stokes equations},
  note={Preprint},
  pages={arXiv:2603.03666},
  date={2026},
}

\bib{Che-Luo-21}{article}{
  author={Cheskidov, Alexey},
   author={Luo, Xiaoyutao},
   title={Sharp nonuniqueness for the Navier-Stokes equations},
   journal={Invent. Math.},
   volume={229},
   date={2022},
   pages={987--1054},
}

\bib{Che-Luo-22-AnnPDE}{article}{
   author={Cheskidov, Alexey},
   author={Luo, Xiaoyutao},
   title={$L^2$-critical nonuniqueness for the 2D Navier-Stokes equations},
   journal={Ann. PDE},
   volume={9},
   date={2023},
   pages={Paper No. 13, 56},
}

\bib{arXiv:2503.05692}{article}{
  author={Cheskidov, Alexey},
  author={Zeng, Zirong},
  author={Zhang, Deng},
  title={Global dissipative solutions of the 3D Navier--Stokes and MHD equations},
  note={Preprint},
  pages={arXiv:2503.05692},
  date={2025},
}

\bib{arXiv:2511.09556}{article}{
  author={Cheskidov, Alexey},
  author={Dai, Mimi},
  author={Palasek, Stan},
  title={Instantaneous Type I blow-up and non-uniqueness of Navier--Stokes flows},
  note={Preprint},
  pages={arXiv:2511.09556},
  date={2025},
}

\bib{Coi-Pal-25}{article}{
    author={Coiculescu, Matei P.},
    author={Palasek, Stan},
    title={Non-uniqueness of smooth solutions of the Navier--Stokes equations from critical data},
    journal={Invent. math.},
    date={2025},
}

\bib{Daneri-Szekelyhidi-2017-ARMA}{article}{
   author={Daneri, Sara},
   author={Sz\'ekelyhidi, L\'aszl\'o, Jr.},
   title={Non-uniqueness and h-principle for H\"older-continuous weak
   solutions of the Euler equations},
   journal={Arch. Ration. Mech. Anal.},
   volume={224},
   date={2017},
   pages={471--514},
}

\bib{DeLellis-Szekelyhidi-09}{article}{
   author={De Lellis, Camillo},
   author={Sz\'ekelyhidi, L\'aszl\'o, Jr.},
   title={The Euler equations as a differential inclusion},
   journal={Ann. of Math. (2)},
   volume={170},
   date={2009},
   pages={1417--1436},
}

\bib{Lellis-Szekelyhidi-2013-Invent}{article}{
   author={De Lellis, Camillo},
   author={Sz\'ekelyhidi, L\'aszl\'o, Jr.},
   title={Dissipative continuous Euler flows},
   journal={Invent. Math.},
   volume={193},
   date={2013},
   pages={377--407},
}

\bib{Escauriaza-Seregin-Sverak-2003}{article}{
   author={Eskauriaza, L.},
   author={Ser\"egin, G. A.},
   author={Shverak, V.},
   title={$L_{3,\infty}$-solutions of Navier-Stokes equations and backward
   uniqueness},
   language={Russian, with Russian summary},
   journal={Uspekhi Mat. Nauk},
   volume={58},
   date={2003},
   pages={3--44},
   translation={
      journal={Russian Math. Surveys},
      volume={58},
      date={2003},
      pages={211--250},
   },
}


\bib{Fuj-24}{article}{
   author={Fujii, Mikihiro},
   title={Ill-Posedness of the Two-Dimensional Stationary Navier--Stokes
   Equations on the Whole Plane},
   journal={Ann. PDE},
   volume={10},
   date={2024},
   pages={Paper No. 10},
}


\bib{Fuj-AIPHC}{article}{
  author={Fujii, Mikihiro},
  title={Time-periodic solutions to the Navier--Stokes equations on the whole space including the two-dimensional case},
  journal={Ann. Inst. H. Poincaré C Anal. Non Linéaire},
  date={2025},
  note={published online first},
}

\bib{arXiv:2509.21272v2}{article}{
  author={Fujii, Mikihiro},
  author={Tsurumi, Hiroyuki},
  title={Asymptotic instability for the forced {Navier--Stokes} equations in critical {Besov} spaces},
  note={Preprint},
  pages={arXiv:2509.21272v2},
  date={2025},
}

\bib{Fuj-Kato-64-ARMA}{article}{
   author={Fujita, Hiroshi},
   author={Kato, Tosio},
   title={On the Navier-Stokes initial value problem. I},
   journal={Arch. Rational Mech. Anal.},
   volume={16},
   date={1964},
   pages={269--315},
}

\bib{Fur-Rie-Ter-2000}{article}{
   author={Furioli, Giulia},
   author={Lemari\'e-Rieusset, Pierre G.},
   author={Terraneo, Elide},
   title={Unicit\'e{} dans $L^3(\R^3)$ et d'autres espaces fonctionnels
   limites pour Navier-Stokes},
   language={French, with English and French summaries},
   journal={Rev. Mat. Iberoamericana},
   volume={16},
   date={2000},
   pages={605--667},
}

\bib{Gallagher-Iftimie-Planchon-2002}{article}{
   author={Gallagher, Isabelle},
   author={Iftimie, Drago\c s},
   author={Planchon, Fabrice},
   title={Asymptotics and stability for global solutions to the
   Navier-Stokes equations},
   language={English, with English and French summaries},
   conference={
      title={Journ\'ees ``\'Equations aux D\'eriv\'ees Partielles''},
      address={Forges-les-Eaux},
      date={2002},
   },
   book={
      publisher={Univ. Nantes, Nantes},
   },
   isbn={2-86939-188-9},
   date={2002},
   pages={Exp. No. VI, 9},
}

\bib{Gallagher-Iftimie-Planchon-2003}{article}{
   author={Gallagher, I.},
   author={Iftimie, D.},
   author={Planchon, F.},
   title={Asymptotics and stability for global solutions to the
   Navier-Stokes equations},
   language={English, with English and French summaries},
   journal={Ann. Inst. Fourier (Grenoble)},
   volume={53},
   date={2003},
   pages={1387--1424},
}

\bib{Gig-Miy-85}{article}{
   author={Giga, Yoshikazu},
   author={Miyakawa, Tetsuro},
   title={Solutions in $L_r$ of the Navier-Stokes initial value problem},
   journal={Arch. Rational Mech. Anal.},
   volume={89},
   date={1985},
   pages={267--281},
}

\bib{2509.25116v1}{article}{
  author={Hou, Thomas},
  author={Wang, Yixuan},
  author={Yang, Changhe},
  title={Nonuniqueness of Leray--Hopf solutions to the unforced incompressible 3D Navier--Stokes equation},
  note={Preprint},
  pages={arXiv:2509.25116v1},
  date={2025},
}

\bib{Isett-18}{article}{
   author={Isett, Philip},
   title={A proof of Onsager's conjecture},
   journal={Ann. of Math. (2)},
   volume={188},
   date={2018},
   pages={871--963},
}

\bib{IN}{article}{
	   author={Iwabuchi, Tsukasa},
	   author={Nakamura, Makoto},
	   title={Small solutions for nonlinear heat equations, the Navier-Stokes
	   equation, and the Keller-Segel system in Besov and Triebel-Lizorkin
	   spaces},
	   journal={Adv. Differential Equations},
	   volume={18},
	   date={2013},
	   pages={687--736},
	}

\bib{Iwabuchi-Okazaki-2025}{article}{
  author={Iwabuchi, Tsukasa},
  author={Okazaki, Taiki},
  title={On the uniqueness of the mild solution of the critical quasi-geostrophic equation},
  journal={Commun. Math. Sci.},
  volume={23},
  date={2025},
}

\bib{Jia-Sve-Invent}{article}{
   author={Jia, Hao},
   author={\v Sver\'ak, Vladim\'ir},
   title={Local-in-space estimates near initial time for weak solutions of
   the Navier-Stokes equations and forward self-similar solutions},
   journal={Invent. Math.},
   volume={196},
   date={2014},
   pages={233--265},
}

\bib{Jia-Sve-JFA}{article}{
   author={Jia, Hao},
   author={Sverak, Vladimir},
   title={Are the incompressible 3d Navier-Stokes equations locally
   ill-posed in the natural energy space?},
   journal={J. Funct. Anal.},
   volume={268},
   date={2015},
   pages={3734--3766},
}


\bib{Kan-Koz-Shi-19}{article}{
   author={Kaneko, Kenta},
   author={Kozono, Hideo},
   author={Shimizu, Senjo},
   title={Stationary solution to the Navier-Stokes equations in the scaling
   invariant Besov space and its regularity},
   journal={Indiana Univ. Math. J.},
   volume={68},
   date={2019},
   pages={857--880},
}

\bib{Kato-84}{article}{
   author={Kato, Tosio},
   title={Strong $L^p$-solutions of the Navier-Stokes equation in ${\bf
   R}^{m}$, with applications to weak solutions},
   journal={Math. Z.},
   volume={187},
   date={1984},
   pages={471--480},
}

\bib{Koh-Tat-01}{article}{
   author={Koch, Herbert},
   author={Tataru, Daniel},
   title={Well-posedness for the Navier-Stokes equations},
   journal={Adv. Math.},
   volume={157},
   date={2001},
   pages={22--35},
}

\bib{Kozono-Shimizu-2018}{article}{
   author={Kozono, Hideo},
   author={Shimizu, Senjo},
   title={Navier-Stokes equations with external forces in time-weighted
   Besov spaces},
   journal={Math. Nachr.},
   volume={291},
   date={2018},
   pages={1781--1800},
}

\bib{Kozono-Sohr-1996}{article}{
  author={Kozono, Hideo},
  author={Sohr, Hermann},
  title={Remark on uniqueness of weak solutions to the {Navier--Stokes} equations},
  journal={Analysis},
  volume={16},
  date={1996},
  pages={255--272},
}

\bib{Koz-Yam-94}{article}{
   author={Kozono, Hideo},
   author={Yamazaki, Masao},
   title={Semilinear heat equations and the Navier-Stokes equation with
   distributions in new function spaces as initial data},
   journal={Comm. Partial Differential Equations},
   volume={19},
   date={1994},
   pages={959--1014},
}

\bib{Rieusset-25-JFA}{article}{
   author={Lemari\'e-Rieusset, Pierre Gilles},
   title={Highly singular (frequentially sparse) steady solutions for the 2D
   Navier-Stokes equations on the torus},
   journal={J. Funct. Anal.},
   volume={288},
   date={2025},
   pages={Paper No. 110761, 16},
}


\bib{Leray-1934-Acta}{article}{
   author={Leray, Jean},
   title={Sur le mouvement d'un liquide visqueux emplissant l'espace},
   language={French},
   journal={Acta Math.},
   volume={63},
   date={1934},
   pages={193--248},
}

\bib{Li-Yu-Zhu-25}{article}{
  author={Li, Jinlu},
  author={Yu, Yanghai},
  author={Zhu, Weipeng},
  title={Ill-posedness for the stationary {Navier--Stokes} equations in critical {Besov} spaces},
  journal={SIAM J. Math. Anal.},
  volume={57},
  date={2025},
  pages={2363--2383},
}

\bib{Lio-Mas}{article}{
   author={Lions, P.-L.},
   author={Masmoudi, N.},
   title={Uniqueness of mild solutions of the Navier-Stokes system in $L^N$},
   journal={Comm. Partial Differential Equations},
   volume={26},
   date={2001},
   pages={2211--2226},
}

\bib{Luo-ARMA-2019}{article}{
   author={Luo, Xiaoyutao},
   title={Stationary solutions and nonuniqueness of weak solutions for the
   Navier-Stokes equations in high dimensions},
   journal={Arch. Ration. Mech. Anal.},
   volume={233},
   date={2019},
   pages={701--747},
}

\bib{Masuda-1984}{article}{
   author={Masuda, Ky\=uya},
   title={Weak solutions of Navier-Stokes equations},
   journal={Tohoku Math. J. (2)},
   volume={36},
   date={1984},
   pages={623--646},
}

\bib{Miu-2005}{article}{
   author={Miura, Hideyuki},
   title={Remark on uniqueness of mild solutions to the Navier-Stokes
   equations},
   journal={J. Funct. Anal.},
   volume={218},
   date={2005},
   pages={110--129},
}

\bib{Nakasato-25-JDE}{article}{
   author={Nakasato, Ryosuke},
   title={Analyticity and asymptotic behavior of solutions to the
   Navier-Stokes equations in an end-point scaling critical framework},
   journal={J. Differential Equations},
   volume={453},
   date={2026},
   pages={Paper No. 113851, 29},
}

\bib{Nash-54}{article}{
   author={Nash, J.},
   title={$C^1$ isometric imbeddings},
   journal={Ann. of Math. (2)},
   volume={60},
   date={1954},
   pages={383--396},
}

\bib{Pal-25-IMRN}{article}{
   author={Palasek, Stan},
   title={Non-uniqueness in the Leray-Hopf class for a dyadic Navier-Stokes
   model},
   journal={Int. Math. Res. Not. IMRN},
   date={2025},
   pages={Paper No. rnaf344, 32},
}

\bib{Pla-98}{article}{
   author={Planchon, Fabrice},
   title={Asymptotic behavior of global solutions to the Navier-Stokes
   equations in ${\bf R}^3$},
   journal={Rev. Mat. Iberoamericana},
   volume={14},
   date={1998},
   pages={71--93},
}


\bib{Saw-18}{book}{
   author={Sawano, Yoshihiro},
   title={Theory of Besov spaces},
   series={Developments in Mathematics},
   volume={56},
   publisher={Springer, Singapore},
   date={2018},
   pages={xxiii+945},
}

\bib{Serrin-1963}{article}{
   author={Serrin, James},
   title={The initial value problem for the Navier-Stokes equations},
   conference={
      title={Nonlinear Problems},
      address={Proc. Sympos., Madison, Wis.},
      date={1962},
   },
   book={
      publisher={Univ. Wisconsin Press, Madison, WI},
   },
   date={1963},
   pages={69--98},
}

\bib{Takeuchi-2025}{article}{
  author={Takeuchi, Taiki},
  title={Decay rates of small global solutions to the Navier–Stokes system in scaling invariant homogeneous Besov spaces},
   journal={Math. Ann.},
   date={2025},
}

\bib{Takeuchi-2025-SIMA}{article}{
   author={Takeuchi, Taiki},
   title={On the strong solutions to the Navier-Stokes system with
   excessively singular external forces},
   journal={SIAM J. Math. Anal.},
   volume={57},
   date={2025},
   pages={5382--5419},
}


\bib{Tan-Tsu-Zha-25}{article}{
   author={Tan, Jin},
   author={Tsurumi, Hiroyuki},
   author={Zhang, Xin},
   title={On steady solutions of the Hall-MHD system in Besov spaces},
   journal={Z. Angew. Math. Phys.},
   volume={76},
   date={2025},
   pages={Paper No. 139, 21},
}

\bib{Tsu-19-ARMA}{article}{
   author={Tsurumi, Hiroyuki},
   title={Well-posedness and ill-posedness problems of the stationary
   Navier-Stokes equations in scaling invariant Besov spaces},
   journal={Arch. Ration. Mech. Anal.},
   volume={234},
   date={2019},
   pages={911--923},
}

\bib{Wan-15}{article}{
   author={Wang, Baoxiang},
   title={Ill-posedness for the Navier-Stokes equations in critical Besov
   spaces $\dot B_{\infty,q}^{-1}$},
   journal={Adv. Math.},
   volume={268},
   date={2015},
   pages={350--372},
}

\bib{Watanabe-2025}{article}{
  author={Watanabe, Keiichi},
  title={Algebraic decay rates for strong solutions to the Navier--Stokes equations in critical spaces},
  note={Preprint},
  date={2025},
}

\bib{Yam-2000}{article}{
   author={Yamazaki, Masao},
   title={The Navier-Stokes equations in the weak-$L^n$ space with
   time-dependent external force},
   journal={Math. Ann.},
   volume={317},
   date={2000},
   pages={635--675},
}

\bib{Yon-10}{article}{
   author={Yoneda, Tsuyoshi},
   title={Ill-posedness of the 3D-Navier-Stokes equations in a generalized
   Besov space near $\rm BMO^{-1}$},
   journal={J. Funct. Anal.},
   volume={258},
   date={2010},
   pages={3376--3387},
}

\bib{arXiv:2402.01174v2}{article}{
  author={Zhan, Zhirun},
  title={Uniqueness of mild solutions to the Navier--Stokes equations in weak-type $L^d$ space},
  note={Preprint},
  pages={arXiv:2402.01174v2},
  date={2024},
}


\end{thebibliography}
\end{document}